\documentclass[reqno,11pt]{amsart}

\usepackage{xcolor}

\usepackage{graphicx}
\usepackage[hidelinks]{hyperref}
\usepackage[shortlabels]{enumitem}
\usepackage{mathrsfs}
\usepackage[latin1]{inputenc}

\newtheorem{theorem}{Theorem}[section]
\newtheorem{lemma}[theorem]{Lemma}
\newtheorem{corollary}[theorem]{Corollary}
\newtheorem{proposition}[theorem]{Proposition}
\newtheorem{example}[theorem]{Example}

\theoremstyle{definition}
\newtheorem{remark}[theorem]{{\bf Remark}}
\newtheorem{definition}[theorem]{Definition}

\newcommand{\cc}{\mathbb{C}}
\newcommand{\hh}{\mathbb{H}}
\newcommand{\nn}{\mathbb{N}}
\newcommand{\rr}{\mathbb{R}}

\DeclareMathOperator{\tr}{Tr}

\DeclareMathOperator{\ddet}{Ddet}
\DeclareMathOperator{\sdet}{Sdet}
\DeclareMathOperator{\ber}{Ber}
\newcommand\bigw{\scalebox{.95}[1]{$\bigwedge$}}

\renewcommand{\Re}{\mathrm{Re}}
\renewcommand{\Im}{\mathrm{Im}}

\makeatletter
    \def\overbracket#1{\mathop{\vbox{\ialign{##\crcr\noalign{\kern3\p@}
          \downbracketfill\crcr\noalign{\kern3\p@\nointerlineskip}
          $\hfil\displaystyle{#1}\hfil$\crcr}}}\limits}
    \def\downbracketfill{$\m@th
      \makesm@sh{\llap{\vrule\@height.7\p@\@depth2.3\p@\@width.7\p@}}%
      \leaders\vrule\@height.7\p@\hfill
      \makesm@sh{\rlap{\vrule\@height.7\p@\@depth2.3\p@\@width.7\p@}}$}
\makeatother

\newcommand\blfootnote[1]{%
	\begingroup
	\renewcommand\thefootnote{}\footnote{#1}%
	\addtocounter{footnote}{-1}%
	\endgroup
}

\title[Nuclearity and Grothendieck-Lidskii formula for quaternionic operators]
{Nuclearity and Grothendieck-Lidskii formula for quaternionic operators}
\oddsidemargin
0.2in \evensidemargin 0.2in \topmargin -0.5in \textwidth
15.5truecm \textheight 23truecm

\author[P. Cerejeiras]{P. Cerejeiras}
\address{(PC) CIDMA, Departamento de Matem\'atica\\
Universidade de Aveiro\\
Campus de Santiago\\
P-3810-193, Aveiro\\
Portugal}
\email{pceres@ua.pt}
\email{}

\author[F. Colombo]{F. Colombo}
\address{(FC) Politecnico di
Milano\\Dipartimento di Matematica\\Via E. Bonardi, 9\\20133 Milano\\Italy}
\email{fabrizio.colombo@polimi.it}

\author[A. Debernardi Pinos]{A. Debernardi Pinos}
\address{(ADP)
Universitat Aut\`onoma de Barcelona\\
Departament de Matem\`atiques\\
Campus de Bellaterra, Edifici C\\
08193 Bellaterra (Barcelona)\\
Spain}
\email{adebernardipinos@gmail.com}

\author[U. K\"ahler]{U. K\"ahler}
\address{(UK) CIDMA, Departamento de Matem\'atica\\
Universidade de Aveiro\\
Campus de Santiago\\
P-3810-193, Aveiro\\
Portugal}
\email{ukaehler@ua.pt}

\author[I. Sabadini]{I. Sabadini}
\address{(IS) Politecnico di
Milano\\Dipartimento di Matematica\\Via E. Bonardi, 9\\20133 Milano\\Italy}
\email{irene.sabadini@polimi.it}

\begin{document}

\begin{abstract}
	We introduce an appropriate notion of trace in the setting of quaternionic linear operators, arising from the well-known companion matrices. We then use this notion to define the quaternionic Fredholm determinant of trace-class operators in Hilbert spaces, and show that an analog of the classical Grothendieck-Lidskii formula, relating the trace of an operator with its eigenvalues, holds. We then extend these results to the so-called $\frac{2}{3}$-nuclear (Fredholm) operators in the context of quaternionic locally convex spaces. While doing so, we develop some results in the theory of topological tensor products of noncommutative modules, and show that the trace defined ad hoc in terms of companion matrices, arises naturally as part of a canonical trace.
\end{abstract}

\blfootnote{\textbf{Acknowledgments}: P. Cerejeiras and U. K\"ahler were supported by CIDMA, through the Portuguese FCT (\href{https://doi.org/10.54499/UIDP/04106/2020}{https://doi.org/10.54499/UIDP/04106/2020} and \href{https://doi.org/10.54499/UIDB/04106/2020}{https://doi.org/10.54499/UIDB/04106/2020}). A. Debernardi Pinos was partially supported by the Catalan AGAUR, through the Beatriu de Pin\'os program (2021BP  00072), and by CIDMA, through the Portuguese FCT (\href{https://doi.org/10.54499/UIDP/04106/2020}{https://doi.org/10.54499/UIDP/04106/2020} and \href{https://doi.org/10.54499/UIDB/04106/2020}{https://doi.org/10.54499/UIDB/04106/2020}). U. K\"ahler thanks Politecnico di Milano for the support and kind hospitality during the period in which part of this paper was written.
}
	
\maketitle

MSC2020 classification, primary: 46S05, 47S05. Secondary: 47B10.


\section{Introduction}

In this paper we aim to obtain extensions of the Grothendieck-Lidskii trace formula to quaternionic (right-linear) $p$-summable Fredholm operators in locally convex spaces (or to trace-class operators in the Hilbert space case). The problem of finding a corresponding formula for the Fredholm determinant, which is often included as part of the Grothendieck-Lidskii formula, is also addressed.

Given an operator $T$ belonging to an appropriate class, in the commutative setting, these formulae read as
\begin{equation}
\label{EQ1}
\tr T = \sum_{k=1}^\infty \lambda_k(T),\qquad \det(I+T)= \prod_{k=1}^\infty (1+\lambda_k(T)),
\end{equation}
where $\{\lambda_k(T)\}$ is the  sequence of eigenvalues of $T$ and $\tr$, $\det$ denote the trace and the determinant of $T$, respectively, and $I$ denotes the identity operator. We emphasize that the operator trace is defined in terms of a chosen basis of the underlying vector space. Thus, the above formulae are independent of the choice of the basis, since the sequence $\{\lambda_k(T)\}$ is invariant with respect to such a choice.

Recently, some efforts have been made in order to extend the Grothendieck-Lidskii trace formula to different settings, such as compact Lie groups \cite{DR} and  variable exponent Lebesgue spaces \cite{DRW} (without considering the problem of the Fredholm determinant).

Trying to directly extend these formulae to the quaternionic setting leads to intrinsic limitations that arise in this context, even for finite-rank operators. A well known issue is that if $\lambda\in \hh$ is a right-eigenvalue of an operator $T$, then so is $s^{-1}\lambda s$, for any $s\in\hh$ with $s\not=0$. Indeed, the equation $Tx=x\lambda$ implies
\begin{equation}
\label{EQ2}
T(xs)= ( x\lambda ) s =(xs)(s^{-1}\lambda s).
\end{equation}
This means two things. First of all, we have an infinite number of eigenvalues. However, there are only $n$ equivalence classes if we make the identification $\mu \sim \lambda$ if and only if there exists $s \in \hh$ with $s\not=0$ such that $\mu = s^{-1}\lambda s.$ Secondly, while $\lambda$ is an eigenvalue with eigenvector $x$, we have that  $\mu$ is an eigenvalue associated with  the eigenvector $xs.$ Thus, the trace of a quaternionic linear operator need not to be equal to the sum of its eigenvalues as in \eqref{EQ1} (in fact, none of those quantities, nor the Fredholm determinant, are well defined).

In order to overcome these issues, I. Gelfand, S. Gelfand, Retakh, and Wilson developed in \cite{quasidet} the general  theory of quasideterminants. Since a functional satisfying all the classical axioms of the determinant cannot exist over the quaternions  \cite{CDL}, it is natural to consider quasideterminants as a natural equivalent. Quasideterminants were first introduced in \cite{GR} for matrices over free division rings. The main difference with the commutative determinant is that the former is the analog of a ratio between the determinant of an $n\times n$ matrix and that of an $(n-1)\times(n-1)$ submatrix. This means that in such a general setting, quasideterminants are not polynomial functions of the matrix entries (as in the commutative case), but rather rational functions.

In this context, the aforementioned authors were able to define in \cite{quasidet} the determinant of a linear map $A:R^m\to R^m$, where $R$ is only required to be an algebra with a unit. The definition is based on a characteristic equation, whose construction is as follows: a vector $v\in R^m$ is said to be cyclic if $\{v,Av,\ldots, A^{m-1}v\}$ is a basis of $R^m$, when viewed as a right $R$-module. In this case, the authors show that, given a cyclic vector $v\in R^m$, there exist elements $\alpha_k=\alpha_k(A,v)\in R$, $k=1,\ldots, m$, such that
\begin{equation}\label{EQgeneralchareq}
(-1)^m v\alpha_m + (-1)^{m-1}(Av)\alpha_{m-1}+\cdots -A^{m-1}v \alpha_1+A^mv=0.
\end{equation}
This equation plays the role of the characteristic polynomial. In the case where $R$ is a commutative algebra, the element $\alpha_m$ is the usual determinant of $A$, i.e. $\alpha_m:=\det A$. Moreover, by the Cayley-Hamilton theorem, the numbers $\alpha_k$ are precisely the coefficients of the characteristic polynomial of $A$. In \cite{quasidet} it is also shown that if the determinant $\alpha_m$ is zero, the corresponding map $A$ is not invertible, which is a natural property one would expect from any  notion of generalized determinant.

However, when $R$ is not commutative, identity \eqref{EQgeneralchareq} presents two limitations that, in general, one would like to avoid. First of all, the numbers $\alpha_k(A,v)$ depend on the choice of the cyclic vector $v$. This is unavoidable, since, while one has equivalence classes of eigenvalues, these do not preserve the eigenspaces. Secondly, in order to compute these elements, the formulae given in  \cite{quasidet} require a priori knowledge on the eigenvalues of $A$. This is in sharp contrast with the commutative (say, complex) setting, where the characteristic polynomial can be computed directly from the entries of the matrix representing the map $A$, and it also contains the information concerning the eigenvalues of $A$.

Therefore, from the abstract point of view \eqref{EQgeneralchareq} is ideal, but it is not suitable for our purposes.

The Dieudonn\'e determinant \cite{D} (defined for general skew-fields and denoted by $\ddet$), introduced by Dieudonn\'e himself in 1943, is, to this day, one of the most widely used alternatives of the classical determinant in noncommutative settings. The properties of such a determinant are close to those of the commutative ones. What is more, in the quaternionic case,  the quasideterminant as defined in \cite{GR} coincides precisely with $\ddet$ (cf. \cite[Section 3]{GGK}).

In relation to the Dieudonn\'e determinant, Study \cite{S} introduced a quaternionic determinant, denoted by $\sdet$, which satisfies 
\begin{equation}\label{sdetddet}
	\ddet A= (\sdet A)^2,
\end{equation}
for a quaternionic matrix $A$. This notion was shown to correspond precisely with the determinant of the so-called companion matrix of $A$ \cite{Al,zhang}.

Another important alternative of the determinant is the so-called \textit{Berezinian} (or also superdeterminant) of a matrix, denoted by $\ber$, see \cite{Ber}. Such a concept is defined for matrices with entries in graded algebras that can be decomposed into an even and an odd part (as, for instance, Clifford algebras, and in particular the quaternions). As the usual determinant, the Berezinian is a multiplicative functional, i.e., $\ber(AB)=\ber(A)\ber(B)$, and furthermore, it satisfies a Liouville-type formula, which in the classical case reads as $\det(\mathrm{exp}(A))=\mathrm{exp}(\tr A)$. However, if we consider a complex matrix defined by square blocks as
\[
A=\begin{pmatrix}
C_1&0\\
0& C_2\\
\end{pmatrix},
\]
we have that $\ber(A)=\det(C_1)\det(C_2)^{-1}$ (i.e., invertibility of $C_2$ is required). This makes the Berezinian not suitable as a replacement of the determinant, for instance, when relating the latter with the eigenvalues of $A$. However, it has been proved very useful in problems of noncommutative geometry that are related to physics \cite{COP,Wi}.

For further alternative definitions of determinants of quaternionic matrices see the survey \cite{Al}.

All the above considerations imply that there is not a clear notion of what the Fredholm determinant of a quaternionic operator should be.

Similar limitations appear when trying to define a trace of a (finite-rank) operator $T$ over the quaternions. In complex Hilbert spaces $H$, one has
\[
\tr T = \sum_{k=1}^n\langle Te_k,e_k\rangle=\sum_{k=1}^n \lambda_k(T),
\]
for \textit{any} orthonormal basis $\{e_k\}$ of $H$. As observed above, in the quaternionic case none of these two sums is well defined and, in principle, the former depends on the choice of $\{e_k\}$. All in all, the concept of trace of a quaternionic operator has similar issues as the determinant: no functional $\tr$ satisfying all axioms of the usual trace may exist. In fact, the existence of such a functional would imply the existence of a determinant.

Nevertheless, different extensions of the trace that serve different purposes may be found in the literature. For instance, the concept of ``supertrace'' is very useful in the context of noncommutative geometry \cite{Ber} (in fact, the supertrace and the aforementioned Berezinian are connected in a similar way as the determinant and the trace are, through Liouville's formula). It is clear, due to the relation between the supertrace and the Berezinian, that such a concept is not suitable for our purposes.
For the sake of completeness, we recall that there is another way to define the trace, due to Dixmier (see \cite{dixmier} and \cite{CSdix}), for operators not of trace class, but this is beyond the scope of this paper.

Before describing the approach we take in order to overcome the above mentioned difficulties, we observe the following: in view of \eqref{EQ2}, when speaking about (right-) eigenvalues of quaternionic operators, one should consider equivalence classes $[\lambda]$. It is well known that these classes of equivalence possess two invariant quantities, namely $\Re (\lambda)$ and $|\Im (\lambda)|$ (or, alternatively, $|\lambda|$). In other words, if $\lambda',\lambda''\in[\lambda]$, then $\Re (\lambda')=\Re(\lambda'')$ and $|\Im(\lambda')|	=|\Im(\lambda'')|$, i.e., $|\lambda'|	=|\lambda''|$.

\medskip
In view of this observation, it seems natural to consider an approach based on the concept of  $S$-spectrum, whose existence was suggested by the formulation of quaternionic quantum mechanics.
Despite the fact that quaternionic quantum mechanics was proposed by G. Birkhoff and  J. von Neumann in 1936 \cite{BF},  the $S$-spectrum of a quaternionic linear operator $T$ was introduced only in 2006, and is defined as
\[
\sigma_S(T)=\big\{\lambda\in \hh: T^2- 2 {\rm Re}(\lambda)T+|\lambda|^2I \text{ is not invertible as a bounded operator}\big\}.
\]
The discovery of $S$-spectrum and of its related $S$-functional calculus is well explained in the introduction of the book \cite{6CKG} with a complete list of the references and it is also described how hypercomplex analysis methods were used to identify the appropriate notion of  quaternionic spectrum. For the Clifford setting, see \cite{CSS}.
We mention that the spectral theory on the $S$-spectrum has several applications, for
example, to fractional diffusion problems \cite{6CG}. For further applications see \cite{BARCCO} and the references therein.
Moreover, using the $S$-spectrum it is possible to define several types of
functional calculi, based on suitable integral transforms,
 for functions that include axially harmonic and axially polyanalytic functions (see
\cite{quatAX} and \cite{quatPoly} for the quaternionic setting), while using the
poly slice monogenic Cauchy formulae, it was possible to define the polyanalytic functional calculus
in \cite{PSFUNC}.

\medskip
The $S$-spectrum can be defined for operators over more general noncommutative structures, but in the quaternionic case, it has the particularity that the point $S$-spectrum coincides with the right spectrum (which is not true in more general settings). Furthermore, the $S$-spectrum can be described in terms of the aforementioned invariants of each equivalence class $[\lambda]$, namely $\Re(\lambda)$ and $|\Im(\lambda)|$.  In other words, the $S$-spectrum contains the invariants from all the equivalence classes of eigenvalues $[\lambda]$ of $T$. This makes the $S$-spectrum the appropriate tool for studying quaternionic operators.

The notion of $S$-spectrum has been successfully used in recent research by a number of authors, showing its potential as the counterpart of the usual spectrum in the noncommutative setting, cf., for instance, \cite{ACKS,6SpecThm1,CK,CS,CSS,CSSbook}.

In order to study spectral properties of quaternionic operators, we develop a characteristic equation  that is connected to the $S$-spectrum. For practical purposes, such an equation corresponds (in finite dimensions) to the characteristic equation of the well-known \textit{companion matrix} of a quaternionic matrix \cite{lee,zhang}.

In the finite-dimensional (matrix) setting, the invariants we consider are those arising from the companion matrix. In particular, we introduce the first and second order traces of $A\in M_n(\hh)$ (more generally, we introduce the $k$-th order trace $T_{\hh,k}(A)$, for all $k\geq 0$), which are well defined and have the properties that
\[
T_{\hh,1}(A)= 2\Re\bigg(\sum_{k=1}^n \lambda_{k}\bigg), \qquad T_{\hh,2}(A) = \sum_{\ell=1}^n |\lambda_\ell|^2 +4\sum_{\ell=1}^{n-1}\Re(\lambda_\ell)\bigg(\sum_{m=\ell+1}^{n}\Re(\lambda_m)\bigg),
\]
where $\lambda_1,\ldots , \lambda_n$ are the \textit{standard eigenvalues} of $A$ (i.e., they are the unique representative in each equivalence class $[\lambda_k]$ with the property that $\lambda_k\in \cc$ and $\Im \lambda_k\geq 0$). We stress that the first-order trace does not involve all the invariants of the equivalence classes of eigenvalues since only the real parts are involved. Thus, considering the second-order trace as a complement to the missing information in $T_{\hh,1}(A)$ is necessary.

We then relate the definition of these $k$-th order traces with finite-rank quaternionic operators $T$ in Banach spaces. More importantly, we actually prove that these quantities do not depend of the choice of the basis in which the corresponding operator is represented, i.e., they represent \textit{invariant} quantities associated with  those operators. This allows us to define a quaternionic Fredholm determinant, and prove that
\[
\det{}_\hh(I+T)=\sum_{k=0}^{2n} T_{\hh,k}(T).
\]

The above constructions ultimately allow us to prove a quaternionic Grothendieck-Lidskii-type formula for trace-class operators in Hilbert spaces, and for $2/3$-summable Fredholm operators in locally convex spaces (which reduce to $2/3$-nuclear operators in Banach spaces),
\begin{align*}
T_{\hh,1}(T)&= 2\Re\bigg(\sum_{k=1}^\infty \lambda_{k}(T)\bigg),\qquad \text{and}\\
\det{}_\hh(I+T)&=\sum_{k=0}^\infty T_{\hh,k}(T)=\prod_{k=1}^\infty(1+2\Re(\lambda_k(T))+|\lambda_k(T)|^2),
\end{align*}
where $\{\lambda_n(T)\}$ are the standard eigenvalues of the corresponding operator $T$ (compare with \eqref{EQ1}).

The definition of a quaternionic trace by means of the usual trace of the associated companion matrix seems ad hoc, at a first glance. This is because it only involves the real parts of the inner products that form the corresponding operator matrix. However, as we show in Section~\ref{SUBSECcanonicaltrace}, this definition of the trace is actually the natural one, as it is precisely the canonical linear form arising in the tensor product of a quaternionic vector space and its dual (similar to the classical trace, which is the canonical linear form in the tensor product of a vector space and its dual).

The outline of the paper is as follows. In Section~\ref{SECprelim} we briefly introduce the needed material concerning quaternionic vector spaces and tensor products of quaternionic spaces, as well as some useful facts about entire functions. Section~\ref{SECmatrices} is devoted to introducing the invariants that we study in finite-dimensional quaternionic vector spaces, in particular the first and second order traces, and to finding the desired relations and formulas relating those invariants, the standard eigenvalues of quaternionic maps, and the maps themselves. In Section~\ref{SECsvd} we study the quaternionic Fredholm determinant of trace-class maps in Hilbert spaces and derive several results that are analogous to those in the classical case, as the analyticity of the determinant, or the Grothendieck-Lidskii formula, among others. Although these results are of interest by themselves, they are also needed in the subsequent parts. Section~\ref{SECnuclear} is devoted to briefly discussing a simplified Grothendieck-Lidskii formula in the quaternionic Banach spaces, which we study in full detail in Section~\ref{SEClcs}, in the context of quaternionic locally convex spaces. To this end, we develop a parallel theory of tensor products of quaternionic spaces as in \cite{GroPTT}, which gives rise to the identification of the quaternionic trace as the canonical form in the tensor product (Subsection~\ref{SUBSECcanonicaltrace}), as described above. After developing such a theory, we give a version of the Grothendieck-Lidskii formula for $\frac{2}{3}$-Fredholm operators in locally convex spaces. In the end, the reader may find some auxiliary results in the appendix.

It is worth mentioning that in the sequel we will work in the setting of quaternionic right/left vector spaces (i.e., with scalar multiplication from the right/left). Formally, these are right/left $\hh$-modules (we use this fact to construct their tensor products in Section~\ref{SECprelim}). The ``module'' nomenclature is also standard in the literature, although in the context of this paper, it seems more natural to use the ``space'' terminology.

\section{Preliminaries}\label{SECprelim}

\subsection{Quaternionic vector spaces}
We denote by $\hh$ the algebra of quaternions \cite{Jo}, where the imaginary units are denoted by $i$ and $j$ (the product of these two gives the independent imaginary unit $ij$). The imaginary units satisfy $i^2=j^2=(ij)^2=-1$. Any element $a\in\hh$ is therefore written as
\[
a=a_0+ ia_i+ja_j+ij a_{ij},
\]
where all coefficients are real-valued. The real and imaginary parts of $a$ are defined, respectively, as
\[
\Re(a)=a_0,\qquad \Im(a)= ia_i+ja_j+ij a_{ij}.
\]
Conjugation is defined for quaternions similarly as in the complex case. More precisely,
\[
\overline{a}= a_0- ia_i-ja_j-ij a_{ij},
\]
(note that in particular, $\overline{i}=-i$, $\overline{j}=-j$, and $\overline{ij}=\overline{j} \,\overline{i}=-ij$).

The set $a\in\hh$ such that $\Re(a)=0$ of the so-called purely imaginary quaternions contains the subset
$$
\mathbb S=\{a\in\hh\ :\ \Re(a)=0,\ |\Im(a)|=1\}.
$$
An element $a$ belongs to $\mathbb S$ if and only if $a^2=-1$. Any non real quaternion $\lambda$ can be uniquely written as $\lambda= \Re(\lambda)+J|\Im(\lambda)|$, $J\in\mathbb S$, whereas $\lambda\in\mathbb R$ can be written as $\lambda= \lambda +J\cdot 0$ for any $J\in\mathbb S$.

In what follows, unless otherwise explicitly stated, the considered quaternionic vector spaces $E$ will be \textit{right} vector spaces, i.e., the scalar multiplication by an element $q\in \hh\backslash\{0\}$ will be a right group action of $\hh\backslash\{0\}$ on $E$. The subspaces of $E$ naturally inherit this structure. On the other hand, all the spaces will be assumed to have a $\rr$-linear structure (i.e., multiplication by real scalars is commutative). It is worth mentioning that there exist quaternionic vector spaces with both left and right $\hh$-linear structure. However, for our purpose, it is enough to require the more common right $\hh$-linear structure.

Finally, all the vector spaces of right-linear operators acting on $E$ will be endowed with a $\rr$-linear structure (usually, such spaces are endowed with a left $\hh$-linear structure, but this already forces $E$ to be left $\hh$-linear, which we do not require).

In finite dimensions, we denote by $\hh^n$ the quaternionic vector space of $n$ dimensions (note that $\hh^n$ also has a natural left $\hh$-linear structure). Its dual space is the left $\hh$-vector space $\overline{\hh^n}$, where the scalar multiplication is given by $a*v=\overline{a}v$ (and $\overline{a}v$ is understood as an element of $\hh^n$). The canonical dual pairing of $\overline{\hh^n}$ and $\hh^n$ defines a left and right $\hh$-linear form $\langle \cdot,\cdot \rangle$ satisfying
\[
\langle a*v,w*b\rangle =\overline{a}\langle v,w\rangle b,
\]
where the left $\hh$-linearity is with respect to the scalar multiplication in $\overline{\hh^n}$ (equivalently, we say that that $\langle \cdot ,\cdot \rangle$ is a sesquilinear form in $\hh^n$). Such a sesquilinear form is actually the canonical inner product in $\hh^n$ given by
\[
\langle w,v\rangle =\sum_{\ell=1}^{n} \overline{w_\ell}v_\ell,
\]
where $v=(v_1,\ldots ,v_n)$ and $w=(w_1,\ldots, w_n)$. As in the complex case, we have $\langle v,v\rangle\geq 0$ (with equality if and only if $v=0$), and $\langle v,w\rangle = \overline{\langle w,v\rangle}$.

We call a basis $\{e_1,\ldots ,e_n\}$ of $\hh^n$ (which exists; cf. \cite[Ch. I]{BoTVS}) orthonormal if $\langle e_\ell,e_m\rangle=\delta_{\ell m}$ for every $\ell,m$, where $\delta_{\ell m}$ is the Kronecker delta.

\subsection{Entire functions}

We will need some useful facts about entire functions. We start with the following result on infinite products.
\begin{theorem}\cite[Theorem 2.6.5]{boas}\label{THMinfprod}
	Let $\{\nu_n\}\subset \cc$, and assume that $\big\{\nu_n\big\}\in \ell^p$ for some $0<p<\infty$. Then the infinite product
	\[
	P(z)=\prod_{n=1}^\infty\big(1-\nu_nz\big)
	\]
	is an entire function of genus $0$ and order $p$.
\end{theorem}
Note that the sequence $\{\nu_n\}$ consists precisely of inverses of the zeros of $P$ (counting multiplicity).

The following result relates the order of an entire function and the decay of the coefficients of its representing power series.
\begin{theorem}\cite[Theorem 2.2.2]{boas}\label{THMorder}
	If $f$ is an entire function with $f(z)=\sum_{n=0}^\infty  a_nz^n$, then $f$ is of finite order if and only if
	\[
	\mu=\limsup_n \frac{n\log n}{\log \big(|a_n|^{-1}\big)}
	\]
	is finite. In this case the order of $f$ is equal to $\mu$.
\end{theorem}

We will need the following corollary of Theorem~\ref{THMorder}.
\begin{corollary}\label{REMcorasymptotic}
	Let $\{\nu_n\}\in \ell^p$, and assume that
	\[
	f(z)=\prod_{n=1}^\infty \big( 1+ \nu_n z\big) = \sum_{n=0}^\infty a_nz^n.
	\]
	Then, the coefficients $a_n$ satisfy the estimate $|a_n|\leq Cn^{-\frac{n}{q}}$ for every $q>p$.
\end{corollary}
\begin{proof}
	The expansion of $f$ in the corresponding power series, together with Theorem~\ref{THMorder}, yield
	\[
	\limsup_n \frac{n\log n}{\log \big(|a_n|^{-1}\big)}\leq p.
	\]
	This implies that for every $q>p$, there exists $n_0\in \nn$ such that for every $n\geq n_0$ one has
	\[
	\frac{n\log n}{\log \big(|a_n|^{-1}\big)}\leq q,
	\]
	or equivalently,
	\[
	|a_n|\leq n^{-\frac{n}{q}}.\qedhere
	\]
\end{proof}

\subsection{Tensor products of $\Bbb H$-vector spaces}

Due to the noncommutative nature of quaternionic vector spaces, it is not possible to define their tensor products as in the classical case. We follow the definition from \cite{jacobson}, which is given in the context of noncommutative modules, although we adapt it to the case of vector spaces (as mentioned above, right and left $\hh$-vector spaces can also be regarded as left and $\hh$-modules, respectively). Given right and left $\Bbb H$-vector spaces $E$ and $F$, respectively, the tensor product $E\otimes F$ is defined as the free Abelian group whose generators are $x\otimes y$, with $x\in E$ and $y\in F$, satisfying the relations
\begin{align*}
(x_1+x_2)\otimes y&=x_1\otimes y+x_2\otimes y,\\
x\otimes (y_1+y_2)&=x\otimes y_1+x\otimes y_2,\\
(xq)\otimes y&=x\otimes(qy),\qquad q\in \Bbb H.
\end{align*}
Under these relations, $E\otimes F$ satisfies the usual universal property of tensor products \cite[Ch. 3]{jacobson}. Since we are assuming that $E$ and $F$ are $\rr$-linear vector spaces, we can (and we will) endow $E\otimes F$ with a $\rr$-linear structure via the relations
\[
r(x\otimes y)=(rx)\otimes y =x\otimes (ry)=(x\otimes y)r, \qquad r\in \rr.
\]
Tensor products naturally describe right-linear maps in a right $\hh$-vector space $E$ as follows: to a tensor $x\otimes x'\in E\otimes E'$ corresponds the map (of rank 1)
\[
v\mapsto x\langle x',v\rangle.
\]

A balanced $\rr$-bilinear map from $E\times F$ into a set $M$ (typically a vector space) is a map
\begin{align*}
	\Phi:E\times F&\to M,\\
	(x,y)&\mapsto \Phi( x,y),
\end{align*}
satisfying the properties
\begin{enumerate}
	\item $\Phi( x_1+x_2,y)=\Phi( x_1,y)+\Phi( x_2,y)$ and $\Phi( x,y_1+y_2)=\Phi( x,y_1)+\Phi( x,y_2)$ (additivity),
	\item $\Phi( x q,y) = \Phi(x,qy)$, $q\in \Bbb H$ (balance property),
	\item $r\Phi(x,y)=\Phi(rx,y)=\Phi(x,ry)=\Phi(x,y)r$, $r\in \rr$ ($\rr$-bilinearity).
\end{enumerate}

Note that there is a canonical balanced $\rr$-bilinear map $E\times F\to E\otimes F$ defined by $(x,y)\mapsto x\otimes y$.

A balanced $\rr$-bilinear map on $E\times F$ with $M=\rr$ will be called a balanced $\rr$-bilinear form. 

In the classical case, the trace is the linear form on $E\otimes E'$ induced by the canonical bilinear form on $E\times E'$. We can characterize in a similar manner the trace in $E\otimes E'$ (where the dual $E'$ naturally carries a left $\hh$-linear structure). As many structures are built on the real part of a quaternion (for example, the real inner product is the real part of a quaternionic bilinear form), we show that this is also the case here.
\begin{theorem}
	Any balanced $\rr$-bilinear form $\Phi:E\times E'\to \rr$ is the real part of a left and right $\hh$-linear form $\Psi:E'\times E\to \hh$.
\end{theorem}
\begin{proof}
	Given $\Phi$, define
	\[
		\Psi(x',x)=\Phi(x,x')- \Phi(xi,x')i-\Phi(xj,x')j-\Phi(xij,x')ij.
	\]
	Since $\Phi$ takes values in $\rr$, it is clear that under this definition, $\Phi(x,x')=\Re (\Psi(x',x))$. Further, $\Psi$ is additive in both arguments, since $\Phi$ is. Let us now check left and right linearity. For $q=q_0+q_1i+q_2j+q_3ij\in \hh$, we have, by the additivity and $\rr$-bilinearity of $\Phi$,
	\begin{align*}
\Psi(x',xq)&=\Phi(xq,x')- \Phi(xqi,x')i-\Phi(xqj,x')j-\Phi(xqij,x')ij\\
			&= \Phi(x,x')q_0+\Phi(xi,x')q_1+\Phi(xj,x')q_2+\Phi(xij,x')q_3\\
			&\phantom{=}-\Phi(xi,x')q_0i+\Phi(x,x')q_1 i+\Phi(xij,x')q_2i-\Phi(xj,x')q_3i\\
			&\phantom{=}-\Phi(xj,x')q_0j-\Phi(xij,x')q_1j+\Phi(x,x')q_2j+\Phi(xi,x')q_3j\\
			&\phantom{=}-\Phi(xij,x')q_0ij+\Phi(xj,x')q_1ij-\Phi(xi,x')q_2ij +\Phi(x,x')q_3ij\\
			&=\Phi(x,x')q_0+\Phi(x,x')q_1 i+\Phi(x,x')q_2j+\Phi(x,x')q_3ij\\
			&\phantom{=}+\Phi(xi,x')q_1 -\Phi(xi,x')q_0i+\Phi(xi,x')q_3j-\Phi(xi,x')q_2ij\\
			&\phantom{=}+\Phi(xj,x')q_2-\Phi(xj,x')q_3i-\Phi(xj,x')q_0j+\Phi(xj,x')q_1ij\\
			&\phantom{=}+\Phi(xij,x')q_3+\Phi(xij,x')q_2i-\Phi(xij,x')q_1j-\Phi(xij,x')q_0ij\\
			&=\Phi(x,x')q-\Phi(xi,x')iq-\Phi(xj,x')jq-\Phi(xij,x')ijq=\Psi(x',x)q,
	\end{align*}
	i.e., $\Psi$ is right $\hh$-linear. As for left-linearity, we use the balance property of $\Phi$ and obtain
	\begin{align*}
\Psi(qx',x)&=\Phi(x,qx')- \Phi(x,iqx')i-\Phi(x,jqx')j-\Phi(x,ijqx')ij\\
		 &=\Phi(x,x')q-\Phi(xi,x')qi-\Phi(xj,x')qj-\Phi(xij,x')qij=q\Psi(x',x),
	\end{align*}
	where the precise details are analogous to those above.
\end{proof}

In what follows, for a quaternionic vector space $E$, we denote by $\mathcal{L}(E)$ the vector space of continuous right-linear operators acting on $E$ (recall that  we endow $\mathcal{L}(E)$ with a $\rr$-linear structure).

\section{Finite-rank operators in quaternionic vector spaces}\label{SECmatrices}
As is natural, in order to study finite-rank operators in arbitrary quaternionic vector spaces, it is enough to restrict our attention to quaternionic matrices. Indeed, if $A$ is a finite-rank operator (of rank $n$) on a quaternionic vector space $X$, with representation
\begin{equation}
	\label{EQfiniterankop}
A=\sum_{k=1}^n x_k\otimes x_k',\qquad x_k\in X, \quad x_k'=A^*x_k\in X',
\end{equation}
then the matrix corresponding to the linear map $A$ is
\[
\begin{pmatrix}
\langle x_1',x_1\rangle &\langle x_1',x_2\rangle & \cdots& \langle x_1',x_n\rangle \\
\langle x_2',x_1\rangle &\langle x_2',x_2\rangle & \cdots &\langle x_2',x_n\rangle \\
\vdots &\vdots &\ddots &\vdots \\
\langle x_n',x_1\rangle &\langle x_n',x_2\rangle & \cdots &\langle x_n',x_n\rangle
\end{pmatrix}\in M_n(\hh).
\]
We emphasize that the vectors $x_k$ in the above representation are such that $\{x_1,\ldots ,x_n\}$ form a basis for $\textrm{Im} A$.

Before proceeding further, let us mention some basic facts about complex linear maps that will be useful later. Let $B\in M_n(\cc)$. The characteristic polynomial of $B$ is
\[
P_B(z)=\det(Iz-B)=\sum_{k=0}^n (-1)^kc_k z^k,
\]
where
\[
c_k= \tr\Big( \bigw^{n-k}B\Big),
\]
and $\bigw^\ell$ denotes the $\ell$-th exterior product (see, e.g., \cite[Ch. 1, Theorem 7.1]{GGK}), which satisfies
\begin{equation}
	\label{EQtraceextprod}
	\tr\Big( \bigw^{k} B\Big) = \frac{1}{k!}\begin{vmatrix}
		\tr B & k-1 & 0 &\cdots & 0\\
		\tr (B^2) & \tr B & k-2 & \cdots & 0\\
		\vdots & \vdots & \ddots & \ddots &\vdots\\
		\tr (B^{k-1}) & \tr(B^{k-2}) &\cdots & \tr B & 1\\
		\tr (B^k) & \tr (B^{k-1}) & \cdots & \tr(B^2) & \tr(B)
	\end{vmatrix}.
\end{equation}
On the other hand, if we write $B$ in the form
\begin{equation*}
B=\sum_{k=1}^n x_k\otimes x_k',
\end{equation*}
then, as shown in \cite[Ch. I]{GroTdF}, we have
\begin{equation}
\label{EQformulaextprod}
\tr\Big( \bigw^{k} B\Big)  =\frac{1}{k!}  \sum_{i_1,\ldots, i_k=1}^n \det(\langle x_{i_\alpha}',x_{i_\beta}\rangle )_{1\leq \alpha,\beta\leq n}.
\end{equation}

\subsection{Companion matrices and associated invariants}

Let $n\geq 1$ and $A\in M_n(\hh)$. We are interested in finding invariant structures such as quaternionic versions of the trace and the determinant that contain information of $A$ (or more specifically, of its eigenvalues), similar to the usual characteristic polynomial for complex-valued matrices.

As already mentioned in the Introduction, several alternative definitions for these invariants already exist in the literature, although they usually do not possess desirable properties. For instance, the formal extension of the classical determinant (see, e.g., \cite[Section 8]{zhang}) does not guarantee the invertibility of a quaternionic matrix $A$ whenever $\det A\neq 0$ (which already indicates that a characteristic polynomial defined in terms of such a determinant will not have desirable properties either). Thus, it is necessary to define these invariant structures differently.

Given $A\in M_n(\hh)$, we follow \cite{lee,zhang} in decomposing $A=A_1+A_2j$, where $A_1,A_2\in M_n(\cc)$, and defining
\[
\chi_A :=\begin{pmatrix}
	A_1 & A_2\\
	-\overline{A_2} & \overline{A_1}
\end{pmatrix}\in M_{2n}(\cc),
\]
as the \textit{companion matrix} of $A$. Such a matrix captures the noncommutativity nature, as well as essential information of $A$. Indeed, as it is shown in \cite[Theorem 8.1]{zhang}, the matrix $A$ is invertible if and only if $\det \chi_A\neq 0$, and the (quaternionic) eigenvalues of $A$ can be obtained through the characteristic polynomial of $\chi_A$,
\[
P_{\chi_A}(z)=\det( I_{2n}z-\chi_A),\qquad z\in \cc.
\]
More precisely, if $\lambda\in \hh$ is a right eigenvalue of $A$, then so is $s^{-1}\lambda s$ for any $s\in \hh \backslash \{ 0 \}$, since, as it is mentioned in \eqref{EQ2},
\begin{equation*}
A (vs)= v\lambda s = (vs)(s^{-1}\lambda s).
\end{equation*}
Thus, as we shall mention later in Remark~\ref{REMclassequiveigenval}, the right-eigenspace associated with  $\lambda$ is not well defined if $\lambda\not\in\mathbb R$. If one selects an imaginary unit $J\in\mathbb S$, by choosing $s$ above appropriately, one may see that to every right eigenvalue (class) $\lambda$ of $A$, correspond two conjugated eigenvalues $\Re(\lambda)\pm J|\Im(\lambda)|$. Thus the class $[\lambda]$ can be identified with two complex conjugate eigenvalues (say, $\lambda^+$ and $\lambda^-=\overline{\lambda^+}$, with positive and negative imaginary parts, respectively). In particular, the classes of equivalence of the eigenvalues $\lambda$ of $A$ have the quantities $\Re\, \lambda$ and $|\lambda|$ invariant. This pair of complex numbers $\lambda^+$ and $\lambda^-$ are precisely the eigenvalues of $\chi_A$ \cite[Corollary~5.1 and Theorem~8.1]{zhang}.
The case $\lambda\in\mathbb R$ obviously fits this description.

Before proceeding further, we recall that the aforementioned Study determinant $\sdet$  of a quaternionic matrix $A$ corresponds precisely to the usual determinant of the companion matrix $\chi_A$ \cite{Al}. Thus, by \eqref{sdetddet}, the latter is also related to the Dieudonn\'e determinant, and in particular,
\[
\det\chi_A=\sdet A=(\ddet A)^2.
\]

Since $\chi_A$ is a complex matrix, its invariants are well described by the characteristic polynomial $P_{\chi_A}(z)$. As already mentioned, the complex eigenvalues of $A$ are precisely the eigenvalues of $\chi_A$. We may go one step further and associate to $A$ some invariant quantities that actually come from the matrix $\chi_A$. Before doing so, one more remark is in order. Using \eqref{EQtraceextprod} and the property that $\chi_{AB}=\chi_A\chi_B$ for $A,B\in M_n(\hh)$, the coefficients of the characteristic polynomial $P_{\chi_A}$ take the form
\begin{equation*}
\tr\Big( \bigw^{k} \chi_A\Big) = \frac{1}{k!}\begin{vmatrix}
\tr \chi_A & k-1 & 0 &\cdots & 0\\
\tr \chi_{A^2} & \tr \chi_A & k-2 & \cdots & 0\\
\vdots & \vdots & \ddots & \ddots &\vdots\\
\tr \chi_{A^{k-1}} & \tr\chi_{A^{k-2}} &\cdots & \tr \chi_A & 1\\
\tr \chi_{A^k} & \tr \chi_{A^{k-1}} & \cdots & \tr\chi_{A^2} & \tr\chi_A
\end{vmatrix}.
\end{equation*}

\begin{definition}\label{DEFinvariants}	Let $A=\big(a_{\ell m}\big)_{\ell,m=1}^n\in M_n(\hh)$. We define the quaternionic (first-order) trace of $A$ as
	\[
	T_{\hh,1}(A):=\tr\chi_A =2\Re\bigg(\sum_{\ell=1}^na_{\ell\ell}\bigg)
	\]
	Further, for $k\geq 2$, we define the $k$-th order trace of $A$ as
	\[
	T_{\hh,k}(A):= \frac{1}{k!}\begin{vmatrix}
	T_{\hh,1} (A) & k-1 & 0 &\cdots & 0\\
	T_{\hh,1} (A^2)  & T_{\hh,1} (A)& k-2 & \cdots & 0\\
	\vdots & \vdots & \ddots & \ddots &\vdots\\
	T_{\hh,1} (A^{k-1})  & T_{\hh,1}(A^{k-2}) &\cdots & T_{\hh,1} (A) & 1\\
	T_{\hh,1} (A^k)  & T_{\hh,1} (A^{k-1})& \cdots & T_{\hh,1}(A^2) & T_{\hh,1}(A)
	\end{vmatrix},
	\]
	and the 0-order trace as $T_{\hh,0}(A)=1$. The quaternionic (Fredholm) determinant of $A$ is then defined in terms of $T_{\hh,k}(A)$, $k\geq 0$, as
	\[
	\det{}_\hh (I-zA):=\sum_{k=0}^{2n} (-1)^k T_{\hh,k}(A)z^k.
	\]
\end{definition}
We note that in the above definition one has $T_{\hh,k}(A)=0$ for every $k>2n$. This definition of the trace of a quaternionic matrix has been considered in previous works, see, e.g., \cite{DS}.

\begin{remark}Since the first-order trace $T_{\hh,1}(A)$ is defined in terms of the companion matrix $\chi_A$,  $T_{\hh,k}(A)$ and $\det{}_\hh(I-zA)$  are also naturally related to $\chi_A$ (since both of these quantities are defined in terms of $T_{\hh,1}$). For the $k$-th order traces, we have
\[
T_{\hh,k}(A) = \tr\Big( \bigw^{k} \chi_A\Big),
\]
whilst for the determinant, there holds
\[
\det{}_\hh (I-zA)= \det(I_{2n}-z\chi_A).		
\]
We emphasize that, although $T_{\hh,1}(A)$ is defined as $\tr \chi_A$, the quaternionic Fredholm determinant $\det{}_\hh(I-zA)$ is not defined by just the determinant of the corresponding companion matrix, namely $\det{\chi_{I-zA}}$. Unfortunately, such a simple definition would lead to a characteristic polynomial that is not an entire function, since $\det{\chi_{I-zA}}$ is a polynomial both in $z$ and $\overline{z}$. However, this is not the case if we consider the alternative definition for the Fredholm determinant introduced in Definition~\ref{DEFinvariants} (which corresponds precisely to the associated characteristic polynomial of $\chi_A$, $\widetilde{P}_{\chi_A}$). Such an associated characteristic polynomial vanishes at the inverses of the eigenvalues of $\chi_A$, and can easily be related to $P_{\chi_A}$ as follows:
\begin{align*}
\widetilde{P}_{\chi_A}(z) &= \det(I_{2n}-z\chi_A) = \sum_{k=0}^{2n} (-1)^k\tr\Big( \bigw^{n-k} (z\chi_A) \Big) =  \sum_{k=0}^{2n} (-1)^k \tr\Big( \bigw^{n-k} \chi_A\Big)z^{n-k} \nonumber\\
& = \sum_{k=0}^{2n} (-1)^{n-k}\tr\Big( \bigw^{k} \chi_A\Big) z^{k}  =\sum_{k=0}^{2n}(-1)^{n-k}  c_{n-k}z^k,
\end{align*}
where $c_{k}$ are the coefficients of $P_{\chi_A}(z)$, namely $P_{\chi_A}(z) = \sum_{k=0}^{2n}(-1)^k z^k c_k$. Summarizing, by definition, we have
\[
\widetilde{P}_{\chi_A}(z) = \det{}_\hh(I-zA)=\sum_{k=0}^{2n} (-1)^k T_{\hh,k}(A) z^k.
\]
\end{remark}

We have the following properties  relating the invariants in Definition~\ref{DEFinvariants} and the (standard) eigenvalues of $A$ (cf. \cite{zhang}).
\begin{proposition}\label{PROPfirsttrace}
	Let $A\in M_n(\hh)$. Let $\lambda_1,\ldots , \lambda_n$ be the standard eigenvalues of $A$. Then,
	\[
	T_{\hh,1}(A)=2\Re\bigg(\sum_{k=1}^n \lambda_k\bigg).
	\]
	Furthermore, the quaternionic Fredholm determinant $\widetilde{P}_{\chi_A}(z)=\det{}_\hh(I-zA)$ satisfies the identity
	\begin{equation}
	\label{EQcharpolyeigenval}
	\det{}_\hh(I-zA)=  \prod_{k=1}^n \big(1-2\Re(\lambda_k)z + |\lambda_k|^2z^2 \big).
	\end{equation}
\end{proposition}

\begin{remark}
	In the classical case, the trace may be described as the sum of all eigenvalues of a given matrix, thus giving complete information about such eigenvalues. However, in the quaternionic setting, the first-order trace only gives partial information of the eigenvalues of the corresponding matrix, namely of their real part (thus missing information of the imaginary part, or equivalently, of the modulus). The second-order trace contains the information concerning the moduli of the eigenvalues. More precisely, we have the following.
	
\end{remark}

\begin{proposition}\label{lemma3.2}
	Let $A\in M_n(\hh)$ with $n\geq 2$, and let $\lambda_1,\ldots ,\lambda_n$ be the standard eigenvalues of $A$. We have
	\[
	T_{\hh,2}(A) = \sum_{k=1}^n |\lambda_k|^2 +4\sum_{k=1}^{n-1}\sum_{m=k+1}^{n}\Re(\lambda_k)\Re(\lambda_m).
	\]
\end{proposition}
\begin{proof}
	This follows immediately by inspecting the coefficient of the quadratic term in the right side of \eqref{EQcharpolyeigenval}, which is precisely $T_{\hh,2}(A)$.
\end{proof}

Since the second-order trace $T_{\hh,2}(A)$ contains the missing information in the first-order one ($T_{\hh,1}(A)$), it is important to express $T_{\hh,2}(A)$ explicitly in terms of the entries of $A$.
\begin{proposition}\label{PROPtrace}
	Let $A=\big(a_{\ell m}\big)_{\ell,m=1}^n\in M_n(\hh)$. Then
	\[
	T_{\hh,2}(A)=\sum_{\ell=1}^{n}|a_{\ell\ell}|^2+ 4 \sum_{\ell=1}^{n-1}\sum_{m=\ell+1}^{n}\Re(a_{\ell \ell})\Re(a_{mm}) -2\sum_{\ell=1}^{n-1}\sum_{m=\ell+1}^{n} \Re(a_{m\ell}a_{\ell m}).
	\]
\end{proposition}
In order to prove this identity, we first need the following lemma.

\begin{lemma}\label{LEMquadraticterm}
	Let $B=\{b_{\ell m}\}_{\ell,m=1}^n\in M_{n}(\cc)$, with $n\geq 2$. The coefficient of  the quadratic term of the characteristic polynomial $\widetilde{P}_B$ is
	\[
	\sum_{\ell=1}^{n-1}\sum_{m=\ell+1}^n b_{\ell \ell}b_{mm}-\sum_{\ell=1}^{n-1} \sum_{m=\ell+1}^n b_{\ell m}b_{m\ell}.
	\]
\end{lemma}
\begin{proof}
	We proceed by induction on $n$. If $n=2$, the statement is obviously true. Assume that the claim is true for some $n\in \nn$. Denote, for any $1\leq p\leq n+1$, $B_p=\{b_{\ell m}\}_{\ell,m=1}^p$, and by $\beta_{\ell m}$ the $(m,\ell)$-th minor of the matrix $I_{n+1}-B_{n+1}z$. We also note that for any matrix $C\in M_n(\cc)$, the constant term in $\det (I- Cz)$ equals 1. By the induction hypothesis,
	\begin{align}
	\det (I_{n+1}-z B_{n+1}) &= (1- b_{n+1,n+1}z) \beta_{n+1,n+1} - z\sum_{\ell=1}^n (-1)^{\ell+n+1} b_{\ell,n+1}\beta_{\ell,n+1}\nonumber \\
	&=(1-z b_{n+1,n+1}) \det (I_{n}-z B_{n})-z\sum_{\ell=1}^n (-1)^{\ell+n+1} b_{\ell,n+1}\beta_{\ell,n+1} \nonumber\\
	&= \bigg( \sum_{\ell=1}^{n-1}\sum_{m=\ell+1}^n b_{\ell \ell}b_{mm}-\sum_{\ell=1}^{n-1} \sum_{m=\ell+1}^n b_{\ell m}b_{m\ell} +b_{n+1,n+1}\sum_{\ell=1}^n b_{\ell\ell} \bigg)z^2\nonumber\\
	&\phantom{=}-z\sum_{\ell=1}^{n+1}b_{\ell\ell}+Q(z)-z\sum_{\ell=1}^n (-1)^{\ell+n+1} b_{\ell,n+1}\beta_{\ell,n+1}\nonumber\\
	&= \bigg( \sum_{\ell=1}^{n}\sum_{m=\ell+1}^{n+1} b_{\ell \ell}b_{mm}-\sum_{\ell=1}^{n-1} \sum_{m=\ell+1}^n b_{\ell m}b_{m\ell} \bigg)z^2 -z\sum_{\ell=1}^{n+1}b_{\ell\ell}+Q(z)\nonumber\\
	&\phantom{=}-z\sum_{\ell=1}^n (-1)^{\ell+n+1} b_{\ell,n+1}\beta_{\ell,n+1},\label{EQinductioncharpoly}
	\end{align}
	where $Q$ is a polynomial without linear nor quadratic terms. Note that
	\[
	\beta_{\ell,n+1} = -(-1)^{\ell+n}z b_{n+1,\ell}\det(I_{n-1}-z B(\ell,n)) +z R(z),
	\]
	where $B(\ell,n)$ is a submatrix of $B$ depending on $\ell$ and $n$, and $R(z)=R(n,\ell,z)$ is a polynomial with constant term equal to zero. Since the constant term of $\det(I_{n-1}-z B(\ell,n))$ is equal to 1, we deduce that the linear term in the minor $\beta_{\ell,n+1}$ is precisely
	\[
	-(-1)^{\ell+n}z b_{n+1,\ell}.
	\]
	Thus, the coefficient of the quadratic term in the expression
	\[
	-z\sum_{\ell=1}^n (-1)^{\ell+n+1} b_{\ell,n+1}\beta_{\ell,n+1}
	\]
	is equal to
	\[
	-\sum_{\ell=1}^n b_{\ell,n+1}b_{n+1,\ell},
	\]
	which implies that the coefficient of the quadratic term in \eqref{EQinductioncharpoly} equals
	\begin{align*}
	&\phantom{=}\sum_{\ell=1}^{n}\sum_{m=\ell+1}^{n+1} b_{\ell \ell}b_{mm}-\sum_{\ell=1}^{n-1} \sum_{m=\ell+1}^n b_{\ell m}b_{m\ell} - \sum_{\ell=1}^n b_{\ell,n+1}b_{n+1,\ell}\\
	&=\sum_{\ell=1}^{n}\sum_{m=\ell+1}^{n+1} b_{\ell \ell}b_{mm}-\sum_{\ell=1}^{n} \sum_{m=\ell+1}^{n+1} b_{\ell m}b_{m\ell},
	\end{align*}
	as desired.
\end{proof}
\begin{proof}[Proof of Proposition~\ref{PROPtrace}]
	By the definition of $\chi_A=\{\xi_{\ell m}\}_{\ell,m=1}^{2n}\in M_{2n}(\cc)$ and Lemma~\ref{LEMquadraticterm}, the coefficient of the quadratic term of $\widetilde{P}_{\chi_A}(z)=\det (I_{2n}- \chi_Az)$ is
	\begin{equation}
	\label{EQsecondtraceproof}
	\sum_{\ell=1}^{2n-1}\sum_{m=\ell+1}^{2n}\xi_{\ell \ell}\xi_{mm}-\sum_{\ell=1}^{2n-1} \sum_{m=\ell+1}^{2n} \xi_{\ell m}\xi_{m\ell}
	\end{equation}
	We examine each of these sums separately. For the first sum, noting that $\xi_{\ell\ell}=\overline{\xi}_{\ell+n,\ell+n}$ for $1\leq \ell\leq n$, we obtain
	\begin{align*}
	\sum_{\ell=1}^{2n-1}\sum_{m=\ell+1}^{2n}\xi_{\ell \ell}\xi_{mm}& = \sum_{\ell=1}^{n-1}\sum_{m=\ell+1}^{n}\xi_{\ell \ell}\xi_{mm}+\sum_{\ell=1}^{n}\sum_{m=n+1}^{2n}\xi_{\ell \ell}\xi_{mm} +\sum_{\ell=n}^{2n-1}\sum_{m=\ell+1}^{2n}\xi_{\ell \ell}\xi_{mm}\\
	&= \sum_{\ell=1}^{n-1}\sum_{m=\ell+1}^{n}\xi_{\ell \ell}\xi_{mm}+\sum_{\ell=1}^{n}\sum_{m=1}^{n}\xi_{\ell \ell}\overline{\xi}_{mm} +\sum_{\ell=1}^{n-1}\sum_{m=\ell+1}^{n}\overline{\xi}_{\ell \ell}\overline{\xi}_{mm}\\
	&=2\Re \bigg(\sum_{\ell=1}^{n-1}\sum_{m=\ell+1}^{n}\xi_{\ell \ell}\xi_{mm}\bigg)+\sum_{\ell=1}^{n}\sum_{m=1}^{n}\xi_{\ell \ell}\overline{\xi}_{mm}.
	\end{align*}
	Since
	\[
	\sum_{\ell=1}^{n}\sum_{m=1}^{n}\xi_{\ell \ell}\overline{\xi}_{mm} = \sum_{\ell=1}^n |\xi_{\ell\ell}|^2 + \sum_{\ell=1}^{n}\sum_{\substack{m=1\\ m\neq\ell}}^{n}\xi_{\ell \ell}\overline{\xi}_{mm}= \sum_{\ell=1}^n |\xi_{\ell\ell}|^2+ 2\Re\bigg(\sum_{\ell=1}^{n-1}\sum_{m=\ell+1}^{n}\xi_{\ell \ell}\overline{\xi}_{mm}\bigg),
	\]
	we have
	\begin{align*}
	\sum_{\ell=1}^{2n-1}\sum_{m=\ell+1}^{2n}\xi_{\ell \ell}\xi_{mm}& =\sum_{\ell=1}^n |\xi_{\ell\ell}|^2 + 2\Re\bigg(\sum_{\ell=1}^{n-1}\sum_{m=\ell+1}^{n}\xi_{\ell \ell}\overline{\xi}_{mm}\bigg)+2\Re \bigg(\sum_{\ell=1}^{n-1}\sum_{m=\ell+1}^{n}\xi_{\ell \ell}\xi_{mm}\bigg)\\
	&= 	\sum_{\ell=1}^n |\xi_{\ell\ell}|^2+4 \sum_{\ell=1}^{n-1}\sum_{m=\ell+1}^{n}\Re(\xi_{\ell \ell})\Re(\xi_{mm}).
	\end{align*}
	For the second sum in \eqref{EQsecondtraceproof}, we can write
	\begin{align*}
	\sum_{\ell=1}^{2n-1} \sum_{m=\ell+1}^{2n} \xi_{\ell m}\xi_{m\ell} &=\sum_{\ell=1}^n\sum_{m=\ell+1}^{2n} \xi_{\ell m}\xi_{m\ell}+\sum_{\ell=n+1}^{2n-1}\sum_{m=\ell+1}^{2n} \xi_{\ell m}\xi_{m\ell}\\
	&=\sum_{\ell=1}^{n-1}\sum_{m=\ell+1}^{n} \xi_{\ell m}\xi_{m\ell}+ \sum_{\ell=1}^{n}\sum_{m=n+1}^{2n}\xi_{\ell m}\xi_{m\ell}+\sum_{\ell=n+1}^{2n-1}\sum_{m=\ell+1}^{2n} \xi_{\ell m}\xi_{m\ell}.
	\end{align*}
	By the construction of the matrix $\chi_A$, one has $\xi_{\ell m}=\overline{\xi}_{\ell+n,m+n}$ for every $1\leq \ell,m\leq n$. Thus,
	\[
	\sum_{\ell=1}^{n-1}\sum_{m=\ell+1}^{n} \xi_{\ell m}\xi_{m\ell}+\sum_{\ell=n+1}^{2n-1}\sum_{m=\ell+1}^{2n} \xi_{\ell m}\xi_{m\ell} = 2\Re \bigg(\sum_{\ell=1}^{n-1}\sum_{m=\ell+1}^{n} \xi_{\ell m}\xi_{m\ell}\bigg).
	\]
	On the other hand, since for $1\leq \ell,m\leq n$ there holds $ \xi_{\ell,m+n}=-\overline{\xi}_{\ell+n,m}$, it follows that
	\begin{align*}
	\sum_{\ell=1}^{n}\sum_{m=n+1}^{2n}\xi_{\ell m}\xi_{m\ell} &= \sum_{\ell=1}^{n}\sum_{m=1}^{n}\xi_{\ell,m+n }\xi_{m+n,\ell} =- \sum_{\ell=1}^{n}\sum_{m=1}^{n}\overline{\xi}_{\ell+n,m }\xi_{m+n,\ell}\\
	&=-\sum_{\ell=1}^n |\xi_{\ell+n,\ell}|^2 - \sum_{\ell=1}^{n}\sum_{\substack{m=1\\ m\neq\ell}}^{n}\overline{\xi}_{\ell+n,m }\xi_{m+n,\ell}\\
	& = -\sum_{\ell=1}^n |\xi_{\ell+n,\ell}|^2 - 2\Re \bigg(\sum_{\ell=1}^{n-1}\sum_{m=\ell+1}^{n}\overline{\xi}_{\ell+n,m }\xi_{m+n,\ell}\bigg).
	\end{align*}
	Collecting all the above identities, we finally obtain that the coefficient of the quadratic term of $\det (I_{2n}- \chi_Az)$ is
	\begin{align*}
	&\phantom{=}\sum_{\ell=1}^n |\xi_{\ell\ell}|^2+4 \sum_{\ell=1}^{n-1}\sum_{m=\ell+1}^{n}\Re(\xi_{\ell \ell})\Re(\xi_{mm})-\bigg(2\Re \bigg(\sum_{\ell=1}^{n-1}\sum_{m=\ell+1}^{n} \xi_{\ell m}\xi_{m\ell} \bigg) \\
	&\phantom{=}-\sum_{\ell=1}^n |\xi_{\ell+n,\ell}|^2 - 2\Re\bigg( \sum_{\ell=1}^{n-1}\sum_{m=\ell+1}^{n}\overline{\xi}_{\ell+n,m }\xi_{m+n,\ell}\bigg)\bigg).
	\end{align*}
	To conclude, we rewrite the last expression  explicitly in terms of the elements of the quaternionic matrix $A$. For each element $a_{\ell m}\in \hh$ of the matrix $A$, write
	\[
	a_{\ell m} = a^{(0)}_{\ell m}+ia^{(i)}_{\ell m}+ja^{(j)}_{\ell m}+ija^{(ij)}_{\ell m},
	\]
	with each of the coefficients being a real number. We note that, by the definition of $\chi_A$,
	\[
	\Re(\xi_{\ell m}\xi_{m\ell}) = a_{m\ell}^{(0)}a_{\ell m}^{(0)}-a_{\ell m}^{(i)}a_{m\ell}^{(i)}, \qquad \Re(\overline{\xi}_{\ell+n,m }\xi_{m+n,\ell} )= a^{(j)}_{\ell m}a^{(j)}_{m\ell }+a^{(ij)}_{\ell m}a^{(ij)}_{m\ell },
	\]
	for any $1\leq m,\ell\leq n$. Thus,
	\begin{align*}
	&\phantom{=}2\Re \bigg(\sum_{\ell=1}^{n-1}\sum_{m=\ell+1}^{n} \xi_{\ell m}\xi_{m\ell} \bigg) -2\Re\bigg( \sum_{\ell=1}^{n-1}\sum_{m=\ell+1}^{n}\overline{\xi}_{\ell+n,m }\xi_{m+n,\ell}\bigg) \\
	&= 2\sum_{\ell=1}^{n-1}\sum_{m=\ell+1}^{n}a_{m\ell}^{(0)}a_{\ell m}^{(0)}-a_{\ell m}^{(i)}a_{m\ell}^{(i)} -a^{(j)}_{\ell m}a^{(j)}_{m\ell }-a^{(ij)}_{\ell m}a^{(ij)}_{m\ell }=2\sum_{\ell=1}^{n-1}\sum_{m=\ell+1}^{n} \Re(a_{m\ell}a_{\ell m}).
	\end{align*}
	On the other hand, if $1\leq \ell\leq n$, we have $\Re \,\xi_{\ell\ell}=\Re \, a_{\ell \ell}$. Finally, noting that for $1\leq \ell\leq n$ there holds
	\[
	|a_{\ell\ell}|^2= |\xi_{\ell\ell}|^2+|\xi_{\ell+n ,\ell}|^2,
	\]
	we finally obtain
	\[
	T_{\hh,2}(A) = \sum_{\ell=1}^{n}|a_{\ell\ell}|^2+ 4 \sum_{\ell=1}^{n-1}\sum_{m=\ell+1}^{n}\Re(a_{\ell \ell})\Re(a_{mm}) -2\sum_{\ell=1}^{n-1}\sum_{m=\ell+1}^{n} \Re(a_{m\ell}a_{\ell m}).\qedhere
	\]
\end{proof}
We conclude this subsection by giving a simple example and computing $T_{\hh,1}(A)$ and $T_{\hh,2}(A)$ directly from the matrix entries.
\begin{example}
	Consider the matrix
	\[
	A=\begin{pmatrix}
	3+i & 0\\
	0& ij
	\end{pmatrix}.
	\]
	The standard eigenvalues of $A$ are $3+i$ and $i$. On the other hand, the  characteristic polynomial of $A$ is
	\[
	\widetilde{P}_{\chi_A}(z)=\det(I-z\chi_A) = 1-6z+11z^2-6z^3+10z^4.
	\]
	We can observe that, indeed, the linear term of such a polynomial carries the coefficient $-2\Re(3+i+ij)=-6$ and $T_{\hh,1}(A)=6$, whilst the quadratic term is
	\[
	T_{\hh,2}(A)= |3+i|^2+|ij|^2+4\Re(3+i)\Re(ij)-2\Re((3+i)(ij))=11.
	\]
\end{example}

\subsection{Representations of the trace and determinant of a quaternionic finite-rank operator}

Given the finite-rank operator $F$ from  \eqref{EQfiniterankop}, we aim to obtain representations for the quaternionic traces $T_{\hh,k}(F)$ and the Fredholm determinant $\det{}_\hh(I-zF)$ in terms of $x_m$ and $x_m'$, $m=1,\ldots ,n$. These representations may be seen as the noncommutative analog of  \eqref{EQformulaextprod}. Obtaining such representations will allow us to extend the definitions of the quaternionic traces and  determinant to the infinite-rank case, for operators of the form
\[
T=\sum_{m=1}^\infty x_m\otimes x_m'=\sum_{m=1}^\infty x_m\langle x_m',\cdot \rangle.
\]
We start by showing the commutativity property of the trace.
\begin{proposition}
	Let $X$ be a quaternionic vector space and $F,G:X\to X$ be finite-rank linear operators \textnormal{(}of the same rank\textnormal{)} of the form
	\[
	F=\sum_{m=1}^n x_m\otimes x_m',\qquad  G=\sum_{\ell=1}^n y_\ell\otimes y_\ell',\quad x_m,\,y_\ell\in X, \quad x_m'=F^*x_m,\,y_\ell'=G^*y_\ell\in X'.
	\]
	Then
	\[
	T_{\hh,1}(FG)=T_{\hh,1}(GF)=\sum_{\ell=1}^n\sum_{m=1}^n  2\Re (\langle y_\ell',x_m\rangle\langle x_m',y_\ell\rangle).
	\]
\end{proposition}
\begin{proof}
	Note that
	\[
	FG=\sum_{\ell=1}^n \bigg(\sum_{m=1}^n x_m\langle x_m',y_\ell\rangle\bigg)\otimes y_\ell',\qquad GF=\sum_{m=1}^n \bigg(\sum_{\ell=1}^n y_\ell\langle y_\ell',x_m \rangle\bigg)\otimes x_m',
	\]
	and hence
	\[
	T_{\hh,1}(FG)=\sum_{\ell=1}^n\sum_{m=1}^n  2\Re (\langle y_\ell',x_m\rangle\langle x_m',y_\ell\rangle) =\sum_{m=1}^n\sum_{\ell=1}^n  2\Re (\langle x_m',y_\ell\rangle\langle y_\ell',x_m\rangle )=T_{\hh,1}(GF),
	\]
	where we have used that $\Re(ab)=\Re(ba)$ for any $a,b\in\hh$.
\end{proof}

\begin{corollary}\label{PROPbasisindep}
	Let $X$ be a quaternionic vector space and $F:X\to X$ be a finite-rank linear operator of the form
	\[
	F=\sum_{m=1}^n x_m\otimes x_m',\qquad x_m\in X, \quad x_m'=F^*x_m\in X'.
	\]
	Then,
	\[
	T_{\hh,1}(F)=2\sum_{m=1}^n \Re(\langle x_m',x_m\rangle),
	\]
	and for $k\geq 2$,
	\[
	T_{\hh,1}(F^k)=2\sum_{m_1=1}^n \cdots \sum_{m_k=1}^n \Re(\langle x_{m_1}',x_{m_2} \rangle \langle x_{m_2}',x_{m_3}\rangle \langle x_{m_3}',x_{m_4}\rangle \cdots  \langle x_{m_{k-1}}',x_{m_k} \rangle\langle x_{m_k}',x_{m_1}\rangle).
	\]
	Furthermore, for $k\geq 1$, $T_{\hh,1}(F^k)$ does not depend on the choice of the vectors $x_m$, $x_m'$.
\end{corollary}
\begin{remark}\label{REMdetpolyn}
In view of Proposition~\ref{PROPbasisindep}, we observe that the identity
\[
 \det{}_\hh(I-zF)=\sum_{k=0}^{2n} (-1)^k T_{\hh,k}(F) z^k
\]
does not depend on the choice of the vectors $x_m$ and $x_m'$ defining $F$, since $T_{\hh,k}(F)$ can be written in terms of the first-order traces $T_{\hh,1}(F^\ell)$, $\ell=1,\ldots,k$.
\end{remark}
\begin{proof}[Proof of Proposition~\ref{PROPbasisindep}]
	Recall that given $F$ in such a form, the matrix representing the linear map is
	\[
	A_F=\begin{pmatrix}
		\langle x_1',x_1\rangle &\langle x_1',x_2\rangle & \cdots& \langle x_1',x_n\rangle \\
		\langle x_2',x_1\rangle &\langle x_2',x_2\rangle & \cdots &\langle x_2',x_n\rangle \\
		\vdots &\vdots &\ddots &\vdots \\
		\langle x_n',x_1\rangle &\langle x_n',x_2\rangle & \cdots &\langle x_n',x_n\rangle
	\end{pmatrix},
	\]
	in which case,
	\[
	T_{\hh,1}(F)=T_{\hh,1}(A_F)=2\sum_{m=1}^n \Re(\langle x_m',x_m\rangle).
	\]
	The independence with respect to the choice of the basis follows from Proposition~\ref{PROPfirsttrace}. In order to prove the identity concerning $T_{\hh,1}(F^k)$, $k\geq 2$, we note that $\langle x_m',x_m\rangle=\langle x_m,Fx_m\rangle$. Using this equality with $F^k$ in place of $F$, together with the above identity, we get
	\begin{equation}
		\label{EQtracepowerF}
	T_{\hh,1}(F^k)=2\sum_{m=1}^n \Re(\langle  x_m,F^kx_m\rangle).
	\end{equation}
	Now, by the additivity and right-linearity of $\langle \cdot,\cdot \rangle$, we have
	\begin{align*}
		\sum_{m=1}^n \langle  x_m,F^kx_m\rangle &= 	\sum_{m=1}^n \langle x_m',F^{k-1}x_m\rangle = \sum_{m_1=1}^n \bigg\langle  x_{m_1}' ,\sum_{m_2=1}^n x_{m_2}\langle x_{m_2},F^{k-1}x_{m_1}\rangle \bigg\rangle\\
		&=\sum_{m_1=1}^n \sum_{m_2=1}^n \langle x_{m_1}',x_{m_2} \rangle \langle x_{m_2},F^{k-1}x_{m_1}\rangle\\
		& =\sum_{m_1=1}^n \sum_{m_2=1}^n \langle x_{m_1}',x_{m_2} \rangle \langle x_{m_2}',F^{k-2}x_{m_1}\rangle \\
		& = \sum_{m_1=1}^n \sum_{m_2=1}^n \langle x_{m_1}',x_{m_2} \rangle\bigg\langle x_{m_2}' , \sum_{m_3=1}^n x_{m_3}\langle x_{m_3},F^{k-2}x_{m_1}\rangle   \bigg\rangle \\
		&= \sum_{m_1=1}^n \sum_{m_2=1}^n\sum_{m_3=1}^n \langle x_{m_1}',x_{m_2} \rangle \langle x_{m_2}',x_{m_3}\rangle \langle x_{m_3}',F^{k-3}x_{m_1}\rangle =\cdots \\
		&=\sum_{m_1=1}^n \cdots \sum_{m_k=1}^n \langle x_{m_1}',x_{m_2} \rangle \langle x_{m_2}',x_{m_3}\rangle \langle x_{m_3}',x_{m_4}\rangle \cdots  \langle x_{m_{k-1}}',x_{m_k} \rangle\langle x_{m_k},Fx_{m_1}\rangle \\
		&=\sum_{m_1=1}^n \cdots \sum_{m_k=1}^n \langle x_{m_1}',x_{m_2} \rangle \langle x_{m_2}',x_{m_3}\rangle \langle x_{m_3}',x_{m_4}\rangle \cdots  \langle x_{m_{k-1}}',x_{m_k} \rangle\langle x_{m_k}',x_{m_1}\rangle.
	\end{align*}
	Substituting this equality on the right side of \eqref{EQtracepowerF} yields the desired result.
\end{proof}

For convenience, we relate the obtained representations for the quaternionic trace and determinant of $F$ with its eigenvalues.
\begin{corollary}\label{CORfiniterank}
	Under the hypotheses of Proposition~\ref{PROPbasisindep}, if $\lambda_1,\ldots,\lambda_n$ denote the standard eigenvalues of $F$, we have
	\begin{align*}
	T_{\hh,1}(F)&=2\Re\bigg(\sum_{k=1}^n \lambda_k\bigg),\\
	T_{\hh,2}(F) & =  \sum_{k=1}^n |\lambda_k|^2 +4\sum_{k=1}^{n-1}\sum_{\ell=k+1}^n\Re(\lambda_k)\Re(\lambda_\ell),
	\end{align*}
	and
	\[
	\det{}_\hh(I-zF)=\sum_{k=0}^{2n} (-1)^k T_{\hh,k}(F) z^k = \prod_{k=1}^n \big(1-2\Re(\lambda_k)z+|\lambda_k|^2z^2\big).
	\]
\end{corollary}
The proof is based on the fact that if $A_F$ denotes the matrix representation of $F$, then the eigenvalues of $\chi_{A_F}$ are $\lambda_1,\overline{\lambda_1},\ldots, \lambda_{k},\overline{\lambda_k}$. The corresponding identities then follow from Propositions~\ref{PROPfirsttrace}~and~\ref{lemma3.2}.

\section{Singular value decompositions and trace-class operators}\label{SECsvd}

\subsection{Hilbert spaces over the quaternions}
Let $H$ be a separable quaternionic Hilbert space, and let $T$ be a compact operator acting on $H$. As it is well known, $T$ is compact if and only if $T^*T$ is compact and if and only if $|T|=(T^*T)^{1/2}$ is compact.

The following proposition follows, since it is true in the normal case as a standard consequence of the spectral theorem.
\begin{proposition}
	Let $T\in \mathcal{L}(H)$ be a self-adjoint, compact operator. Then, there exists a sequence of real numbers $\{\lambda_n\}$ \textnormal{(}tending to $0$\textnormal{)} and an orthonormal set $\{e_n\}$ in $H$ such that
	\begin{equation}\label{TRclassQ}
	Tx=\sum_{n=1}^\infty \lambda_n e_n \langle e_n,x\rangle,
	\end{equation}
	for all $x\in H$.
\end{proposition}

\begin{remark}\label{REMclassequiveigenval}
In general, for equivalence classes of (right) eigenvalues of a quaternionic linear operator, the notion of (right) eigenspace is not invariant. This is a consequence of the fact that if $v$ is an eigenvector related to $\lambda$, for any rotation $s\in\mathbb H$, the vector $vs$ is not an eigenvector associated with $\lambda$  but only with the eigenvalue $s^{-1}\lambda s$, which belongs to the same equivalence class. To get an invariant notion for a given equivalence class $[\lambda]$, one has to consider the pseudo-resolvent operator $Q_\lambda(T)$ and the equation $Q_\lambda(T)v=0$ instead of the right eigenvector equation \cite{6CKG}. This results in a subspace, called $S$-eigenspace, which is invariant for the equivalent class. In the case of a self-adjoint operator $T$ where the eigenvalues are real, the eigenvectors form a vector space which coincides with the $S$-eigenspace. Similarly, an equation of the form \eqref{TRclassQ} is not independent of the choice of the basis, in general, although it is if the eigenvalues are real.
\end{remark}

When the compact operator $T$ is not self-adjoint, then one can consider its polar decomposition (see \cite{Teichmueller} and the recent paper \cite{CK} for the more general Clifford module case). More precisely, one has
\begin{equation*}
T=V|T|, \qquad |T|=(T^*T)^{1/2},
\end{equation*}
where $V$ is a partial isometry, i.e., $\|Vx\|=\|x\|$ for all $x\in{\rm ker}(T)^{\perp}$ and $Vx=0$ for all $x\in{\rm ker}(T)$. Note also that $V$ satisfies $|T|=V^* T$ and $V$, $T$ are uniquely determined.

The operator $|T|$ is positive. Thus, it is self-adjoint, and, in particular, we have
\begin{equation}\label{Tx}
|T|x=\sum_{n=1}^\infty \mu_n e_n \langle e_n,x\rangle
\end{equation}
where the sequence $\{\mu_n\}$ consists of real nonnegative numbers. As it is well known, the numbers $\mu_n$ are precisely the \textit{singular values} of $T$, i.e.,
\[
\mu_n=\mu_n(T)=\inf\big\{\| T-T_n\|:T_n\text{ is a finite-rank operator with }\textrm{rank }T_n< n\big\}.
\]

Letting $\sigma_n=Ve_n$, we have that $\{\sigma_n\}$ is an orthonormal set. Moreover, using \eqref{Tx} and the right-linearity of $V$ we have:
\begin{equation*}
Tx=V|T|x=V\bigg(\sum_{n=1}^\infty \mu_n e_n \langle e_n,x\rangle\bigg)=\sum_{n=1}^\infty  V(\mu_n e_n \langle e_n,x\rangle)=\sum_{n=1}^\infty \mu_n  \sigma_n\langle e_n,x\rangle,
\end{equation*}
and the latter is called the canonical decomposition the compact operator $T$.

\begin{definition}
The space of $p$-\textit{Schatten-von Neumann class operators} $S_p$, $0<p<\infty$, consists of those compact operators $T\in\mathcal{L}(H)$ whose sequence of singular values is in $\ell_p$. For $p=1$ we say that $S_1$ is the trace class and for $p=2$, we say that $S_2$ is the class of Hilbert-Schmidt operators.
\end{definition}

We have the analog of a classical result due to Weyl on estimates of eigenvalues:

\begin{theorem}\label{THMlpineqhilbert}
	Let $1\leq p <\infty$ and $T\in S_p$ acting on a quaternionic Hilbert space $H$.  If $\{\lambda_k\}$ is the sequence of standard eigenvalues of $T$, we have
	\[
	\| \{\lambda_k\}\|_{\ell^p}\leq \| T\|_p.
	\]
\end{theorem}
\begin{proof}
	This result is proved in the same spirit as its commutative counterpart (see, e.g. \cite[Theorem~2.3]{simon}). Instead of giving the full proof, we outline the main steps (with proper references), and while doing so, we point out the differences between the commutative and noncommutative cases.
	
	First, we note that for any orthonormal sets $\{f_n\},\{g_n\}\subset H$, we have
	\begin{equation}
		\label{EQinnerprodest}
		\sum_{n=1}^\infty |\langle f_n,Tg_n\rangle|^p\leq \|T\|_p^p,
	\end{equation}
	(this is a consequence of Lemmas 2.1 and 2.2 in \cite{simon}, which are valid also in the quaternionic case). We now use the fact that a Schur decomposition for the given operator $T$ exists. More precisely, let us denote by $E_T$ the smallest closed linear manifold of $H$ containing all eigenvectors and generalized eigenvectors of $T$ corresponding to each nonzero standard eigenvalue (in general eigenspaces of equivalence classes of quaternionic eigenvalues are not well defined, as emphasized in Remark~\ref{REMclassequiveigenval}. However, these eigenspaces are well defined if in each equivalence class of eigenvalues a representative is chosen; in this case such a representative is the standard eigenvalue). Then, there exists an orthonormal basis $\{\varphi_n\}$ of $E_T$ such that, for every $k\in \nn$,
	\[
	T\varphi_k=\sum_{\ell=1}^k a_{\ell k}\varphi_\ell,\qquad \text{where }a_{\ell\ell}=\lambda_\ell.
	\]
	In other words, when restricted to $E_T$, the operator $T$  (written in terms of the basis $\{\varphi_n\}$) is a triangular operator with the standard eigenvalues in the diagonal. Since $\{\varphi_n\}$ is an orthonormal set, we have, for every $k$,
	\[
	\langle \varphi_k, T\varphi_k\rangle = a_{kk}=\lambda_k,
	\]
	and therefore, applying \eqref{EQinnerprodest} with $\{f_n\}=\{g_n\}=\{\varphi_n\}$, we obtain
	\[
	\sum_{k=1}^\infty |\lambda_k|^p \leq \|T\|_p^p,
	\]
	as desired.
\end{proof}

\subsection{Trace-class operators and the Grothendieck-Lidskii formula}

We now derive some properties for quaternionic trace-class operators (that are well-known in the commutative case), which eventually allow us to obtain a quaternionic Grothendieck-Lidskii formula. We start by proving the independence of the trace with respect to any orthonormal basis.

\begin{proposition}\label{PROPbasisindepinf}
	Let $H$ be a quaternionic Hilbert space and $T:H\to H$ be a trace-class operator with representation
	\[
	T=\sum_{m=1}^\infty x_m \langle x_m',\cdot\rangle,
	\]
	where $\{x_m\}$ is an arbitrary orthonormal basis of $H$ and $x_m'=T^*x_m$. Then, the quaternionic trace
	\[
	T_{\hh,1}(T)=2\Re\bigg(\sum_{m=1}^\infty \langle x_m',x_m\rangle\bigg)
	\]
	does not depend on the choice of the basis $\{x_m\}$.
\end{proposition}
\begin{proof}	
	Consider another orthonormal basis $\{y_k\}$ of $H$. For any $\ell\geq 1$,
	\[
	Ty_\ell= \sum_{m=1}^\infty  Tx_m\langle x_m,y_\ell\rangle.
	\]
	Using the additivity and the right-linearity of $\langle \cdot,\cdot \rangle$, together with the above identity, we get
	\[
	\sum_{\ell=1}^\infty \langle y_\ell',y_\ell\rangle= \sum_{\ell=1}^\infty \langle y_\ell,Ty_\ell\rangle= \sum_{\ell=1}^\infty \bigg\langle   y_\ell,\sum_{m=1}^\infty Tx_m\langle x_m,y_\ell\rangle\bigg\rangle =\sum_{\ell=1}^\infty \sum_{m=1}^\infty \langle y_\ell,Tx_m\rangle \langle x_m,y_\ell\rangle,
	\]
	and the last series converges absolutely, since
	\begin{align*}
	\sum_{\ell=1}^\infty \sum_{m=1}^\infty |\langle y_\ell,Tx_m\rangle \langle x_m,y_\ell\rangle|&\leq \sum_{\ell=1}^\infty\sum_{m=1}^\infty |\langle y_\ell,Tx_m\rangle | \|x_m\|_H\| y_\ell\|_H 	=\sum_{m=1}^\infty\sum_{\ell=1}^\infty  |\langle y_\ell,Tx_m\rangle |\\
	&=\sum_{m=1}^\infty \| Tx_m\|_H^2<\infty.
	\end{align*}
	On the other hand, for $m\geq 1$, we get
	\[
	Tx_m=\sum_{\ell=1}^\infty y_\ell\langle y_\ell,Tx_m\rangle.
	\]
	Then,
	\[
	\sum_{m=1}^\infty \langle x_m',x_m\rangle = \sum_{m=1}^\infty \langle x_m,Tx_m\rangle = \sum_{m=1}^\infty  \bigg\langle  x_m,\sum_{\ell=1}^\infty y_\ell\langle y_\ell,Tx_m\rangle \bigg\rangle = \sum_{m=1}^\infty \sum_{\ell=1}^\infty \langle x_m,y_\ell\rangle \langle y_\ell,Tx_m\rangle.
	\]
	The last series converges absolutely, which is proved similarly as above. We also note that, while the products in the last series do not commute, their real parts do, since $\Re(ab)=\Re(ba)$ for any $a,b\in \hh$. Therefore,
	\begin{align*}
	2\Re \bigg(\sum_{m=1}^\infty \langle x_m',x_m\rangle \bigg) &= 2\Re\bigg(\sum_{m=1}^\infty \sum_{\ell=1}^\infty \langle x_m,y_\ell\rangle \langle y_\ell,Tx_m\rangle\bigg) \\
	&=2\Re\bigg( \sum_{\ell=1}^\infty \sum_{m=1}^\infty \langle y_\ell,Tx_m\rangle \langle x_m,y_\ell\rangle \bigg)=2\Re\bigg( \sum_{\ell=1}^\infty \langle y_\ell',y_\ell\rangle\bigg),
	\end{align*}
	i.e., the definition of $T_{\hh,1}(T)$ does not depend on the choice of the basis $\{x_k\}$.
\end{proof}

\begin{remark}\label{REMtrace}
	If $\{T_n\}$ is a sequence of trace-class operators in a quaternionic Hilbert space $H$ converging to $T\in \mathcal{L}(H)$ in the trace-class norm, then 
	\[
	T_{\hh,1}(T)=\lim_{n\to\infty}T_{\hh,1}(T_n).
	\]
	Indeed, this simply follows from the fact that 
	\[
	|T_{\hh,1}(T_n)-T_{\hh,1}(T)|=|T_{\hh,1}(T_n -T)|\leq 2\| T_n-T\|_1\to 0,\qquad \text{as }n\to\infty.
	\]
\end{remark}

\begin{remark}
	We note that, as in the classical case, the trace-class operators on a quaternionic Hilbert space $H$ form a Banach algebra. In other words, if $T:H\to H$ is a trace-class operator, so is $T^k$, for every $k\geq 1$, and therefore Proposition~\ref{PROPbasisindepinf} implies that
	\[
	T_{\hh,1}(T^k)=2\Re\bigg(\sum_{m=1}^\infty \langle  x_m,T^kx_m\rangle\bigg),
	\]
	for every $k\geq 1$, and such a quaternionic trace does not depend on the choice of the basis $\{x_m\}$.
\end{remark}

The quaternionic Grothendieck-Lidskii formula reads as follows:
\begin{theorem}\label{THMlidskiihilbert}
	Let $H$ be a quaternionic Hilbert space and $T:H\to H$ be a trace-class operator. Then, the (quaternionic) Grothendieck-Lidskii formula holds, i.e., if $\{\lambda_k(T)\}$ is the sequence of standard eigenvalues of $T$ and $\{x_k\}$ is an orthonormal basis of $H$, we have
	\begin{align*}
	T_{\hh,1}(T)&=2\Re \bigg(\sum_{k=1}^\infty \langle x_k,Tx_k\rangle\bigg)= 2\Re \bigg(\sum_{k=1}^\infty \lambda_k(T)\bigg),\\
	T_{\hh,2}(T)&=2\Re \bigg(\sum_{k=1}^\infty \langle x_k,T^2x_k\rangle\bigg)= \sum_{k=1}^\infty |\lambda_k(T)|^2 +4\sum_{k=1}^\infty\sum_{\ell=k+1}^\infty\Re(\lambda_k(T))\Re(\lambda_\ell(T)),
	\end{align*}
	and
	\[
	\det{}_\hh (I+T) =\prod_{k=1}^\infty \big(1+2\Re(\lambda_k(T))+|\lambda_k(T)|^2\big).
	\]
	Furthermore, the quaternionic Fredholm determinant $\det{}_\hh (I+zT)$ is an entire function of order 1 and genus 0.
\end{theorem}
\begin{proof}
	As in the proof of Theorem~\ref{THMlpineqhilbert}, let us denote by $E_T$ the smallest closed linear manifold of $H$ containing all eigenvectors and generalized eigenvectors of $T$ corresponding to each nonzero standard eigenvalue. It is clear that $H=E_T\oplus E_T^\perp$, and that $T$ leaves each of these two subspaces invariant. As in the proof of Theorem~\ref{THMlpineqhilbert}, we use the fact that there exists an orthonormal basis $\{\varphi_n\}$ of $E_T$ such that, for every $k\in \nn$,
	\begin{equation}
		\label{EQdiagonaloperator}
	T\varphi_k=\sum_{\ell=1}^k a_{\ell k}\varphi_\ell,\qquad \text{where }a_{\ell\ell}=\lambda_\ell.
	\end{equation}
	Let $P_n$ denote the projection operator onto $\textrm{span}\{\varphi_1,\ldots , \varphi_n\}$. Let also $\{\psi_k\}$ be an orthonormal basis of $E_T^\perp$, and denote by $Q_n$ the projection operator onto $\textrm{span}\{\psi_1,\ldots , \psi_n\}$. Then (cf. \cite[Ch. IV, \S 11]{GGK}) $(P_n+Q_n)T(P_n+Q_n)\to T$ in the trace-class norm, and hence, by Remark~\ref{REMtrace},
	\begin{align*}
	T_{\hh,1}(T)& = \lim_{n\to \infty} T_{\hh,1}((P_n+Q_n)T(P_n+Q_n))=\lim_{n\to \infty} T_{\hh,1}(P_nTP_n)+\lim_{n\to \infty} T_{\hh,1}(Q_nTQ_n)\\
	&=\lim_{n\to\infty} 2\Re\bigg(\sum_{k=1}^n \langle \varphi_k,T\varphi_k\rangle\bigg)+\lim_{n\to\infty} 2\Re\bigg(\sum_{k=1}^n \langle \psi_k,T\psi_k\rangle\bigg).
	\end{align*}
	By Proposition~\ref{PROPfirsttrace}, we have, for every $n$,
	\[
	 2\Re\bigg(\sum_{k=1}^n \langle \psi_k,T\psi_k\rangle\bigg)=0,
	\]
	since the operator $Q_nTQ_n$ does not have any nonzero eigenvalue. Thus, by \eqref{EQdiagonaloperator},
	\[
	T_{\hh,1}(T)= \lim_{n\to\infty} 2\Re\bigg(\sum_{k=1}^n \langle \varphi_k,T\varphi_k\rangle\bigg)= \lim_{n\to\infty} 2\Re\bigg(\sum_{k=1}^n\lambda_k\bigg) = 2\Re\bigg(\sum_{k=1}^\infty\lambda_k\bigg),
	\]
	and the last series converges absolutely, since $T$ is of trace class and hence, by Theorem~\ref{THMlpineqhilbert}, $\{\lambda_k\}\in \ell^1$. Further, the equality in the statement (for an arbitrary orthonormal basis $\{x_k\}$) holds, by Proposition~\ref{PROPbasisindepinf}. The corresponding identity for $T_{\hh,2}(A)$ will readily follow from the representation of the Fredholm determinant.

	Let us now prove the analyticity of the quaternionic Fredholm determinant. Consider the sequence $\{T_n\}$, where each $T_n$ is the (rank $n$) truncation of $T$ written in the form \eqref{EQdiagonaloperator}, so that $\det{}_\hh(I+zT)=\lim_{n\to\infty} \det(I+zT_n)$ (by Theorem~\ref{THMdetindependentchoice}). It is clear that in this case, the sequence $\{\lambda_k(T_n)\}$ of (standard) eigenvalues of $T_n$ converges to $\{\lambda_k(T)\}$ in the $\ell^1$ norm. Let $\varepsilon>0$ and let $n_0\in \nn$ be such that
	\[
	\sum_{k=1}^\infty |\lambda_k(T)-\lambda_k(T_n)|<\varepsilon,\qquad \text{for every }n\geq n_0.
	\]
	By \eqref{EQcharpolyeigenval}, for $n\geq n_0$, since for every $k\leq n$ we have $\lambda_k(T_n)=\lambda_k(T)$ by construction,
	\begin{align*}
	|\det{}_\hh(I+z T_n)| &\leq \prod_{k=1}^{n} (1+2|\lambda_k(T_n)||z|+|\lambda_k(T_n)|^2|z|^2) \\
	&= \prod_{k=1}^n (1+|\lambda_k(T)||z|)^2 \leq  \prod_{k=1}^n e^{2|\lambda_k(T)||z|} =  e^{2|z|\sum\limits_{k=1}^n |\lambda_k(T)|}\\
	&\leq e^{2|z|\big(\sum\limits_{k=1}^\infty  |\lambda_k(T)|+\varepsilon\big)}.
	\end{align*}
	The last inequality is independent of $n$. Thus,
	\[
	|\det{}_\hh(I+z T)|=\lim_{n\to \infty}|\det{}_\hh(I+z T_n)|\leq e^{c|z|},
	\]
	for some $c>0$. Equivalently, the power series
	\[
	\det{}_\hh(I+z T) = \sum_{k=0}^\infty T_{\hh,k}(T)z^k
	\]
	defines an entire function of order 1, and therefore the sequence of functions $\{\det{}_\hh(I+zT_n)\}$ converges uniformly to $\det{}_\hh(I+z T)$ on every compact set. We now show that the sequence of zeros of $\det{}_\hh(I+z T)$ is precisely
	\begin{equation}
		\label{EQseqofzeros}
	\big\{-\lambda_k(T)^{-1}\big\}\cup\big\{ -\overline{\lambda_k(T)}^{-1}\big\}.
	\end{equation}
	To this end, we need the following result.
	\begin{proposition}\cite[Proposition 6.2.7]{MH}\label{PROPhurwitz}
		Let $\{f_n\}$ be a sequence of analytic functions on a region $\Omega\subset \cc$ converging uniformly to a function $f$ (different from the zero function) on every closed disk of $\Omega$. If $\gamma$ is a closed curve inside of $\Omega$ that passes through no zeros of $f$, then there exists $N\in\nn$ (depending on $\gamma$) such that for every $n\geq N$, $f_n$ and $f$ have the same number of zeros inside of $\gamma$ (counting multiplicities).
	\end{proposition}
	Since $\det{}_\hh(I+z T)$ is an analytic function and is not identically zero, its zeros are isolated. Thus, for every standard eigenvalue $\lambda_k(T)$, we may find a disk $D_\delta^k$ of radius $\delta>0$ centered at $\lambda_k(T)^{-1}$, where the only zero of $f$ in such a disk is precisely $\lambda_k(T)^{-1}$, and similarly for $\overline{\lambda_k(T)}$ (unless $\lambda_k(T)\in \rr$, in which case it is clear that such an eigenvalue has double multiplicity). By Proposition~\ref{PROPhurwitz}, there exists $N\in \nn$ such that for $n\geq N$, the number of zeros of each function $\det{}_\hh(I+zT_n)$ contained in $D_\delta^k$ is equal to the multiplicity of $\lambda_k(T)^{-1}$ as a zero of $\det{}_\hh(I+z T)$. Letting $\delta\to 0$, we see that \eqref{EQseqofzeros} is precisely the sequence of zeros of $\det{}_\hh(I+z T)$, counting multiplicities.
	Thus, by Hadamard's representation, we have
	\begin{align*}
	h(z):=\det{}_\hh(I+z T) &= e^{az}\prod_{k=1}^\infty (1+\lambda_k(T) z)e^{-\lambda_k(T) z}(1+\overline{\lambda_k(T)} z)e^{-\overline{\lambda_k(T)} z} \\
	&=e^{az}\prod_{k=1}^\infty (1+2\Re(\lambda_k(T))z+|\lambda_k(T)|^2z^2)e^{-2\Re(\lambda_k)z} .
	\end{align*}
	By applying the product differentiation rule to entire functions represented by infinite products to the latter expression, we obtain $h'(0)=a$. On the other hand, since we also have
	\[
	h(z)=\sum_{k=0}^\infty T_{\hh,k}(T) z^k,
	\]
	we obtain $h'(0)=T_{\hh,1}(T)=a$. Thus,
	\[
	h(z)=\prod_{k=1}^\infty  (1+2\Re(\lambda_k(T))z+|\lambda_k(T)|^2z^2),
	\]
	or in other words, the quaternionic Fredholm determinant of $T$ has genus zero. We emphasize that the desired identity for $T_{\hh,2}(T)$ follows simply by inspecting the quadratic term in the power series that defines $h(z)$.
\end{proof}

\subsection{Further results for operators in quaternionic Hilbert spaces}

To conclude this section, we give two auxiliary results that may be of independent interest, and certainly useful for the case of quaternionic locally convex spaces, which we investigate in Section~\ref{SEClcs} below.

The following will be useful for controlling the norm in $S_r$ of a composition of two operators represented by infinite matrices. Its commutative counterpart is Lemma~1 in \cite[Ch. II, \S 1]{GroPTT}, and, although the proof for the quaternionic case is similar, we include it here for the sake of completeness.
\begin{lemma}\label{LEMcomposition}
	Let $A=\{ \mu_m d_{mn} \nu_n\}_{m,n=1}^\infty$ be an infinite matrix with quaternion entries, and such that $\{\mu_m\}\in \ell^p$, $\{\nu_n\}\in \ell^q$ for some $0<p,q\leq 2$, and $|d_{mn}|\leq M<\infty$. Then the matrix $A$ defines a Hilbert-Schmidt operator in $\ell^2$, and if $r$ is such that
	\[
	\frac{1}{r} = \frac{1}{p}+\frac{1}{q}-\frac{1}{2},
	\]
	then $A\in S_r$, and the inequality
	\[
	\|A\|_r\leq \|\{\mu_m\}\|_{\ell^p}\|\{\nu_n\}\|_{\ell^q}
	\]
	holds.
\end{lemma}

\begin{proof}
	We define infinite diagonal matrices $B=(b_{mn})_{m,n=1}^\infty$ and $C=(c_{mn})_{n,m=1}^\infty$ with diagonal entries $b_{mm}=\mu_m^{1-\frac{p}{2}}$ and $c_{nn}=\nu_n^{1-\frac{q}{2}}$, respectively. Then, $A= B A'C$, where
	\[
	A'=\{a_{mn}'\}_{m,n=1}^\infty, \quad \text{with} \quad a_{mn}'= \mu_m^\frac{p}{2} d_{mn}\nu_n^{\frac{q}{2}}.
	\]
	The matrix $A'$ defines a Hilbert-Schmidt operator  in $\ell^2$, and $\|A'\|_2^2\leq M\|\{\mu_m\}\|_{\ell^p}^p\|\{\nu_n\}\|_{\ell^q}^q$. Note also that the matrices $B$ and $C$ define operators in $S_{p^*}$ and $S_{q^*}$, respectively, where
	\[
	p^*=\frac{p}{1-p/2}, \quad \text{and}\quad q^*=\frac{q}{1-q/2}.
	\]
	Indeed, this readily follows from the equalities $\|B\|_{p^*}^{p^*}=\|\{\mu_m\}\|_{\ell^p}^p$ and $\|C\|_{q^*}^{q^*}= \|\{\nu_n\}\|_{\ell^q}^q$. Thus, writing
	\[
	\frac{1}{r}=\frac{1}{p^*}+\frac{1}{q^*}+\frac{1}{2}=  \frac{1}{p}+\frac{1}{q}-\frac{1}{2},
	\]
	it follows from H\"older's inequality for operators (see \cite{Horn} for the complex case; the extension to the quaternionic case is straightforward) that
	\begin{align*}
		\|A\|_{r}&\leq \| B\|_{p^*}\| A'\|_{2}\|C\|_{q^*} \leq M\|\{\mu_m\}\|_{\ell^p}^{p/p^*} (\|\{\mu_m\}\|_{\ell^p}^p\|\{\nu_n\}\|_{\ell^q}^q)^{1/2}\|\{\nu_n\}\|_{\ell^q}^{q/q^*}\\
		&=M\|\{\mu_m\}\|_{\ell^p}\|\{\nu_n\}\|_{\ell^q}.\qedhere
	\end{align*}
\end{proof}

As a corollary of Lemma~\ref{LEMcomposition}, we have the following.
\begin{corollary}\label{CORgenuszero}
	Under the assumptions of Lemma~\ref{LEMcomposition}, if $r\leq 1$ (or equivalently, if $1/q+1/p\leq 3/2$), the quaternionic Fredholm determinant of the matrix $A$ is an entire function of order $r$ and genus zero. Furthermore, the sequence of its standard eigenvalues $\{\lambda_k(A)\}$ belongs to $\ell^r$. In particular, the inequality
	\[
	\|\{\lambda_k(A)\}\|_{\ell^r}\leq C\|\{\mu_m\}\|_{\ell^p}\|\{\nu_n\}\|_{\ell^q}.
	\]
	holds.
\end{corollary}
\begin{proof}
	The statement concerning the quaternionic Fredholm determinant of $A$ follows readily from Theorem~\ref{THMlidskiihilbert} (together with Theorem~\ref{THMinfprod}). The inequality involving the sequence of standard eigenvalues is a combination of the inequalities in Theorem~\ref{THMlpineqhilbert} and Lemma~\ref{LEMcomposition}.
\end{proof}

\section{The case of quaternionic Banach spaces}\label{SECnuclear}

\subsection{Nuclear operators in Banach spaces}

Let $X$ be a quaternionic Banach space with the approximation property, i.e., such that for every compact set $K\subset X$ and every $\varepsilon>0$ there exists an operator $F$ of finite rank such that $\|x-Fx\|<\varepsilon$ for all $x\in K$.

Denote by $D_p=D_p(X)$ ($0<p\leq 1$) the set of all right-linear operators $T\in \mathcal{L}(X)$ which allow representations
\begin{equation}\label{nuclear_rep}
Tx=\sum_{k=1}^\infty x_k \langle x_k',x\rangle, \qquad x_k\in X,\quad  x_k'\in X',
\end{equation}
with
\[
\sum_{k=1}^\infty \|x_k\|_X^p \|x_k'\|_{X'}^p<\infty.
\]
We define the semi-norms
\[
\|T\|_p=\inf \bigg( \sum_{k=1}^\infty\|x_k\|_X^p \|x_k'\|_{X'}^p\bigg)^{1/p},
\]
where the infimum is taken over all representations \eqref{nuclear_rep}.  In particular, we have  that
\[
\|T\|_1=\inf  \bigg(\sum_{k=1}^\infty\|x_k\|_X \|x_k'\|_{X'} \bigg).
\]

It can easily be proved that the set $D_p$ is a subspace of $\mathcal{L}(X)$ with the semi-norm $\|\cdot\|_p$, called the space of $p$-nuclear operators. For $p=1$,  $T\in D_1$ is just called a nuclear operator. It is clear that for $p\leq q$, one has $D_p\subset D_q$. Also, $D_1$ is a subalgebra of $\mathcal{L}(X)$ and it coincides with the set of trace-class operators whenever $X$ is a Hilbert space. The space $D_p$ inherits the approximation property.

We first define the trace and the Fredholm determinant of a quaternionic nuclear operator $T:X\to X$ in our setting.
\begin{definition}\label{DEFtracesbanach}
	Let $X$ be a quaternionic Banach space possessing the approximation property and $T\in \mathcal{L}(X)$ be a nuclear operator. For $k\in \nn$, the $k$-th order trace of $T$ is defined as
\begin{equation*}
	T_{\hh,k}(T):= \frac{1}{k!}\begin{vmatrix}
		T_{\hh,1} (T) & k-1 & 0 &\cdots & 0\\
		T_{\hh,1} (T^2)  & T_{\hh,1} (T)& k-2 & \cdots & 0\\
		\vdots & \vdots & \ddots & \ddots &\vdots\\
		T_{\hh,1} (T^{k-1})  & T_{\hh,1}(T^{k-2}) &\cdots & T_{\hh,1} (T) & 1\\
		T_{\hh,1} (T^k)  & T_{\hh,1} (T^{k-1})& \cdots & T_{\hh,1}(T^2) & T_{\hh,1}(T)
	\end{vmatrix},
\end{equation*}	
where
\[
T_{\hh,1}(T)=\sum_{m=1}^\infty 2\Re(\langle x_m',x_m\rangle),
\]
and
\[
T_{\hh,1} (T^k) =\sum_{m_1, \ldots, m_k=1}^\infty  2\Re\big( \langle  x_{m_1}',x_{m_2}\rangle\langle x_{m_2}',x_{m_3}\rangle\cdots \langle  x_{m_k}',x_{m_1}\rangle\big),\quad k\geq 2.
\]
We also formally define $T_{\hh,0}(T)=1$. The (quaternionic) Fredholm determinant of $T$ is defined as
\[
	\det{}_\hh(I+zT):=\sum_{k=0}^\infty T_{\hh,k}(T)z^k ,\qquad z\in \cc.
\]
\end{definition}

In principle, these definitions depend on the choice of the sequences $\{x_k\}$ and $\{x_k'\}$. However, one may prove the following extension of \cite[Ch. V, Theorem 1.2]{GGK} to the quaternionic case.
\begin{theorem}\label{THMapproxbanach}
	Let $X$ be a quaternionic Banach space possessing the approximation property, and let $T\in \mathcal{L}(X)$ be a nuclear operator. Then, the quaternionic trace $T_{\hh,1}(T)$ from Definition~\ref{DEFtracesbanach} does not depend on the choice of the sequences $\{x_k\}\subset X$ and $\{x_k'\}\subset X'$ \textnormal{(}in consequence, $T_{\hh,k}(T)$, $k\geq 2$, and $\det{}_\hh(I+zT)$ do not depend on such sequences either\textnormal{)}. Furthermore, for any net of finite-rank operators $\{R_\ell\}_{\ell\in \Lambda}$ converging to the identity (in the operator norm), we have
	\[
	\lim_{\ell\in \Lambda} T_{\hh,1}(TR_\ell)=T_{\hh,1}(T), \qquad \text{and}\qquad \lim_{\ell\in \Lambda}\det{}_{\hh}(I+TR_\ell)=\det{}_\hh (I+T).
	\]
\end{theorem}
We do not present the proof, as it follows the same lines as the one for the classical case from \cite{GGK}. The conclusion of Theorem~\ref{THMapproxbanach} does not hold if $X$ does not possess the approximation property (cf. \cite{enflo}; see also the discussion in \cite{litvinov}).

As we will see in the general context of locally convex spaces, if $T$ is a $\frac{2}{3}$-nuclear operator, then its trace is well defined without any further assumptions on $X$.

\subsection{The Grothendieck-Lidskii formula in quaternionic Banach spaces}
In this section we use the above results to obtain a Grothendieck-Lidskii-type formula in the context of quaternionic Banach spaces with the approximation property.

We aim to prove that for the class of $\frac{2}{3}$-nuclear operators in quaternionic Banach spaces, the trace and Fredholm determinant from Definition~\ref{DEFtracesbanach} may be expressed in terms of the standard eigenvalues of $T$, just as in the Hilbert space case (hence, they do not depend on the choice of the sequences $\{x_k\}$ and $\{x_k'\}$ in \eqref{nuclear_rep}). We give a simplified formula concerning the quaternionic Fredholm determinant (namely, an equality only for $z=1$), without studying its analyticity properties. We also avoid a treatment through tensor products as in \cite[Ch. V]{GGK}, which require some appropriate topological remarks. These tasks are postponed until the last section, where we give a detailed solution of the problem in full generality, in the context of locally convex spaces.

\begin{theorem}(Grothendieck-Lidskii formula)
	Let $X$ be a quaternionic Banach space possessing the approximation property. If $T\in \mathcal{L}(X)$ is a $\frac{2}{3}$-nuclear operator, and  $\{\lambda_k(T)\}$ is the sequence of standard eigenvalues of $T$, then
	\begin{align*}
			T_{\hh,1}(T)&=2\Re\bigg(\sum_{k=1}^\infty \lambda_k(T) \bigg),\\
			T_{\hh,2}(T)&= \sum_{k=1}^\infty |\lambda_k(T)|^2 +4\sum_{k=1}^\infty\sum_{\ell=k+1}^\infty\Re(\lambda_k(T))\Re(\lambda_\ell(T)).\nonumber
		\end{align*}
	Furthermore, the identity
	\[
	\det{}_\hh(I+T)=\prod_{k=1}^\infty(1+2\Re(\lambda_k(T))+|\lambda_k(T)|^2)=\sum_{k=0}^\infty  T_{\hh,k}(T)
	\]
	holds.
\end{theorem}

\begin{proof}
	Let us consider a net $\{R_\ell\}_{\ell\in \Lambda}$ converging to the identity operator. By Theorem~\ref{THMapproxbanach}, $T_{\hh,1}(T)=\lim_{\ell\in \Lambda}T_{\hh,1}(TR_\ell)$. Now, for each $\ell\in \Lambda$, putting
	\[
	TR_\ell = \sum_{k=1}^{n_\ell} x_{k,\ell}\langle x_{k,\ell}',\cdot \rangle,
	\]
	we have
	\[
	T_{\hh,1}(TR_\ell)=2\Re\bigg(\sum_{k=1}^{n_\ell} \langle x_{k,\ell}', x_{k,\ell}\rangle \bigg),
	\]
	Now we can consider the matrices
	\[
	C_\ell=\mathrm{diag}\big(\beta_1^\ell,\beta_2^\ell,\ldots,\beta_{n_\ell}^\ell\big),
	\]
	with $\beta_k^\ell=\|x_{k,\ell}\|^{1/3}\|x_{k,\ell}'\|^{1/3}$ and
	\[	B_\ell=\bigg(\frac{\|x_{k,\ell}\|^{1/3}\|x_{m,\ell}'\|^{1/3}}{\|x_{m,\ell}\|^{2/3}\|x_{k,\ell}'\|^{2/3}}\langle x_{k,\ell}',x_{m,\ell}\rangle \bigg)_{k,m=1}^{n_\ell}
	\]
	
	Furthermore, the operators given by the infinite matrices
	\[
	C=\mathrm{diag}(\beta_1,\beta_2,\ldots),
	\]
	with $\beta_k=\|x_{k}\|^{1/3}\|x_{k}'\|^{1/3}$, and
	\[
	B=(b_{k m})_{k,m=1}^{\infty}=\bigg(\frac{\|x_{k}\|^{1/3}\|x_{m}'\|^{1/3}}{\|x_{m}\|^{2/3}\|x_{k}'\|^{2/3}}\langle x_{k}',x_{m}\rangle\bigg)_{k,m=1}^{\infty}
	\]
	are (quaternionic) Hilbert-Schmidt operators over $\ell_2$, since $\sum_{k=1}^\infty \beta_k^2<\infty$ and $\sum_{k,m=1}^\infty |b_{km}|^2<\infty$. Thus, the operator $M=CB$ is of trace class in $\ell_2$. Furthermore, defining $M_\ell=C_\ell B_\ell$, we also have, by Theorem~\ref{THMapproxbanach},
	\[
	T_{\hh,1}(M)=\lim_{\ell\in\Lambda}T_{\hh,1}(M_\ell)=\lim_{\ell\in\Lambda}T_{\hh,1}(TR_\ell)=T_{\hh,1}(T).
	\]
	From the finite-dimensional version of Grothendieck-Lidskii's theorem (Corollary~\ref{CORfiniterank}), we get
	\begin{align*}
	T_{\hh,1}(T)&= \lim_{\ell\in\Lambda}T_{\hh,1}(TR_\ell) =\lim_{\ell\in\Lambda} T_{\hh,1}(M_\ell)\\ &=\lim_{\ell\in\Lambda} 2\Re\bigg(\sum_{k=1}^{n_\ell}\lambda_k(TR_\ell)\bigg)
	=2\Re\bigg(\sum_{k=1}^{\infty}\lambda_k(T)\bigg).
	\end{align*}
	The equality concerning $T_{\hh,2}(T)$ is proved similarly.
	
	Now, since $M$ is a trace-class operator in $\ell^2$, we have that $\sum_k |\lambda_k(T)|<\infty$, which of course implies that $\{\lambda_k(T)\}\in \ell^2$. Thus, by Theorem~\ref{THMapproxbanach},
	\begin{align*}
	\det{}_\hh(I+T)&=\lim_{\ell\in \Lambda}\det{}_\hh(I+TR_\ell)=\lim_{i\in\Lambda} \prod_{k=1}^{n_\ell}(1+2\Re(\lambda_k(TR_\ell))+|\lambda_k(TR_\ell)|^2) \\
	&=\prod_{k=1}^\infty(1+2\Re(\lambda_k(T))+|\lambda_k(T)|^2).
	\end{align*}
	It remains to check whether the last infinite product converges. Its convergence is equivalent to the convergence of
	\[
	\sum_{k=1}^\infty \log (1+2\Re(\lambda_k(T))+|\lambda_k(T)|^2).
	\]
	Since $\{\lambda_k(T)\}$ vanishes at infinity, we have, for all $k\geq k_0$,
	\[
	\log (1+2\Re(\lambda_k(T))+|\lambda_k(T)|^2) \approx 2\Re(\lambda_k(T))+|\lambda_k(T)|^2.
	\]
	Hence,
	\[
	\sum_{k=k_0}^\infty \log (1+2\Re(\lambda_k(T))+|\lambda_k(T)|^2)\approx \sum_{k=k_0}^\infty (2\Re(\lambda_k(T))+|\lambda_k(T)|^2),
	\]
	which converges, since $\sum (2\Re(\lambda_k(T))$ converges and $\{\lambda_k(T)\}\in \ell^2$. From the determinant identity it follows that for $z$ small enough, one has
	\[
	\det{}_\hh(I+zT)=\prod_{k=1}^\infty(1+2\Re(\lambda_k(T))z+|\lambda_k(T)|^2z^2),
	\]
	and, by inspecting the coefficient of $z^2$ in this expression (which is precisely $T_{\hh,2}(T)$, by definition), we obtain the desired identity for $T_{\hh,2}(T)$.
\end{proof}

\section{Locally convex spaces over the quaternions and $p$-summable Fredholm operators}\label{SEClcs}
A topological right $\hh$-vector space $E$ is said to be locally convex if the origin has a neighborhood basis consisting of convex sets. In what follows we will always assume that $E$ is also a Hausdorff space. By $E'$ we denote the left $\hh$-vector space that is strongly dual to $E$, and it will be assumed to be locally convex as well.

For a convex bounded set $B\subset E$, $E_B$ denotes the normed subspace of $E$ obtained by endowing the vector space spanned by $B$ with the norm
\[
\|x\|_B = \inf_{x\in t B}|t|.
\]
If $B$ is complete, then $E_B$ is also complete (and thus, a Banach space). 

\subsection{Fredholm and nuclear operators}
We now introduce Fredholm operators and nuclear operators in locally convex spaces and describe (without proofs) some of their most important topological properties (which easily carry over from the classical case). For a general theory on nuclear operators, we refer to \cite[Ch. 47]{treves}. A thorough theory of these operators was developed in \cite{GroPTT}. For a  concise exposition of the most important facts (which we mainly follow here) we refer to \cite{litvinov}.

\begin{definition}\label{DEFfredholmoperator}
Let $E$ be a quaternionic locally convex space. A right-linear operator $T:E\to E$ is called a \textit{Fredholm operator} if it is of the form
	\begin{equation}
	\label{EQnuclearop}
	Tx=\sum_{k=1}^\infty \mu_k x_k \langle x_k',x\rangle,\qquad x\in E,
\end{equation}
where $\{\mu_k\}\in \ell^1$ is a real sequence, and $\{x_k\}$ (resp. $\{x_k'\}$) is contained in an absolutely convex set $B\subset E$ (resp. $B'\subset E'$) such that $E_B$ (resp. $E'_{B'}$) is complete.
\end{definition}
\begin{definition}
Let $E$ be a quaternionic locally convex space. A Fredholm operator $T$ such that in the representation \eqref{EQnuclearop} the sequence $\{x_k'\}$ is equicontinuous is called a \textit{nuclear operator}.
\end{definition}
\begin{remark} Nuclear operators are always compact (thus, they are continuous).
\end{remark}

In the commutative case, it was shown in \cite{litvinov} that the trace of a nuclear operator on a locally convex space possessing the approximation property is well defined (i.e., it does not depend on the representation \eqref{EQnuclearop}). We will obtain a quaternionic version of the Grothendieck-Lidskii for $2/3$-nuclear operators, which in turn shows that such operators have a well-defined trace, even if the underlying space $E$ does not have the approximation property.

\subsection{Tensor products and the canonical balanced form}\label{SUBSECcanonicaltrace}
Given locally convex spaces $E$ and $F$ (which are right and left $\hh$-vector spaces, respectively), we can endow the algebraic tensor product $E\otimes F$ with the strongest locally convex topology so that the canonical balanced $\rr$-bilinear mapping $E\times F \to E\otimes F$ is separately continuous. The completion of $E\otimes F$ in this topology is called the inductive tensor product of $E$ and $F$, and it is denoted by $E \overline{\otimes}F$ (in \cite{treves}, it is also denoted by $E\widehat{\otimes}_\varepsilon F$, where $\varepsilon$ stands for equicontinuous; such a topology can also be defined in terms of convergence in products of equicontinuous sets of $E'$ and $F'$ \cite[Definition~ 43.1]{treves}).

Given another complete locally convex space $G$, associating with every $\rr$-linear continuous mapping $E\overline{\otimes} F\to G$ its composition with the canonical map $E\times F\to E\overline{\otimes} F$ yields a bijection between the continuous $\rr$-linear mappings $E\overline{\otimes} F\to G$ and the separately continuous $\rr$-bilinear balanced mappings $E\times F\to G$.

In the classical case, the trace of a tensor $x\otimes x'\in E\otimes E'$ is defined as the canonical linear form arising from the canonical bilinear form $E\times E'\to \cc$ given by the dual pairing $\langle x',x\rangle=x'(x)$. In the quaternionic case, the dual pairing
\begin{align*}
	E\times E'&\to \hh\\
	(x,x')&\mapsto \langle x',x\rangle=x'(x),
\end{align*}
does not define a balanced form, and, in turn, does not allow to define a canonical $\rr$-linear form on $E\otimes E'$ (the balance property in the definition of the tensor product $E\otimes E'$ requires any such form to be balanced in $E\times E'$ in order to be well defined). However, since for any $a,b\in \hh$ we have $\Re(ab)=\Re(ba)$, the map
\begin{align}
	E\times E'&\to \rr \nonumber\\
	(x,x')&\mapsto 2\Re(\langle x',x\rangle),\label{EQbalancedmap}
\end{align}
defines a $\rr$-bilinear balanced form on $E\times E'$, which we use to define the quaternionic trace (note that the factor $2$ in \eqref{EQbalancedmap} is not necessary, but with this definition the quaternionic trace of a tensor is consistent with the classical trace of the companion matrix of the associated operator, see Subsection~\ref{SUBSECrightlinearop}). Since the mapping \eqref{EQbalancedmap} is separately continuous, it can be extended to a continuous $\rr$-linear form on $E\overline{\otimes}E'$, which we call the quaternionic trace, and denote it by $T_{\hh,1}(\cdot)$. This allows to define higher order traces of tensors $u\in E\overline{\otimes}E'$ as done above. To this end, it is necessary to introduce the composition of two Fredholm kernels $v\circ u$ in the obvious way (so that this composition coincides with the composition of the associated operators). If
\begin{equation}
	\label{EQuv}
u=\sum_{k=1}^\infty \mu_k x_k\otimes x_k',\qquad \text{and}\qquad v=\sum_{\ell=1}^\infty \nu_\ell y_\ell\otimes y_\ell',
\end{equation}
we define
\begin{equation}
	\label{EQuvcomp}
v\circ u=\sum_{k,\ell=1}^\infty \mu_k\nu_\ell  \big(y_\ell\langle y_\ell',x_k\rangle \big)  \otimes x_k',
\end{equation}
which is again a Fredholm kernel. For $k\geq 2$, we define $u^k=u\circ u^{k-1}$, and
\begin{equation}\label{Higher_order_trace_LCS}
	T_{\hh,k}(u)= \frac{1}{k!}\begin{vmatrix}
		T_{\hh,1} (u) & k-1 & 0 &\cdots & 0\\
		T_{\hh,1} (u^2)  & T_{\hh,1} (u)& k-2 & \cdots & 0\\
		\vdots & \vdots & \ddots & \ddots &\vdots\\
		T_{\hh,1} (u^{k-1})  & T_{\hh,1}(u^{k-2}) &\cdots & T_{\hh,1} (u) & 1\\
		T_{\hh,1} (u^k)  & T_{\hh,1} (u^{k-1})& \cdots & T_{\hh,1}(u^2) & T_{\hh,1}(u)
	\end{vmatrix}.
\end{equation}
As usual, we take the convention $T_{\hh,0}(u)=1$. 

After defining the $k$-th order trace of a Fredholm kernel $u$, we may proceed as in the Hilbert space case, and introduce the quaternionic Fredholm determinant of $u$. Such a determinant takes the form
\[
\det{}_{\hh}(I+zu):=\sum_{k=0}^\infty T_{\hh,k}(u)z^k.
\]
The Fredholm kernel $u$ naturally defines a right-linear map, and the above determinant vanishes at (minus) the inverses of the standard eigenvalues (and their conjugates) of such a map. We proceed to discuss in detail how Fredholm kernels and right-linear maps are related.

\subsection{Tensor products, right-linear operators, and Fredholm kernels}
\label{SUBSECrightlinearop}
Let us denote by $\mathcal{L}_w(E)$ the $\rr$-vector space of all weakly continuous right-linear operators in $E$ endowed with the weak operator topology, given by the semi-norms $A\to | \langle x',Ax\rangle|$, with $A\in \mathcal{L}_w(E)$, $x\in E$, and $x'\in E'$. Let us denote by $\Gamma$ the $\rr$-linear mapping $E\otimes E'\to \mathcal{L}_w(E)$, under which the tensor
\[
u=\sum_{k=1}^n \mu_k x_k\otimes x_k'
\]
is mapped into the finite-rank operator
\[
\Gamma(u):\, x\mapsto \sum_{k=1}^n \mu_k x_k \langle x_k',x\rangle,
\]
It is clear that $\Gamma$ establishes a bijection between $E\otimes E'$ and the subspace of operators of finite rank in $R(E)$. Observe that the balance property in $E\otimes E'$ is satisfied  through the map $\Gamma$,
i.e., two equal tensors $(xq)\otimes x'=x\otimes (qx')$ define two equal operators $xq\langle x',\cdot\rangle=x\langle qx',\cdot\rangle$.

\begin{remark}
If $u\in E\otimes E'$, then $T_{\hh,1}(u)$ (the trace form defined through \eqref{EQbalancedmap} and the canonical map $E\times E'\to E\otimes E'$) coincides with the classical trace of the companion matrix of $\Gamma(u)$. More generally, if $u\in E{\otimes}E'$, it is clear that for every $k\geq 0$, one has $T_{\hh,k}(u)=T_{\hh,k}(\Gamma(u))$.
\end{remark}

The map $\Gamma$ is continuous in the inductive topology, but in general it cannot be continued to $E\overline{\otimes} E'$, since $R(E)$ may not be complete. However, there exists a linear subspace of $E\overline{\otimes} E'$ where $\Gamma$ can be continued, which we now introduce.
\begin{definition}
	We denote by $E\overbracket{\otimes} E'$ the subspace of $E\overline{\otimes} E'$ consisting of all elements of the form
	\begin{equation}
		\label{EQfredkernel}
		u=\sum_{k=1}^\infty \mu_k x_k\otimes x_k',
	\end{equation}
	with $\{\mu_k\}$, $\{x_k\}$, and $\{x_k'\}$ as in Definition~\ref{DEFfredholmoperator}. We say that $u\in E\overbracket{\otimes} E'$ is a \textit{Fredholm kernel} on $E$.
\end{definition}
Fredholm kernels were extensively studied by Grothendieck in \cite{GroPTT}. Under the given assumptions, the series \eqref{EQfredkernel} converges in $E\overline{\otimes} E'$. It is clear that if $u$ is a Fredholm kernel, then $\Gamma(u)$ is a Fredholm operator, or in other words, the class of Fredholm operators is the image of all Fredholm kernels under the mapping $\Gamma$ (or rather, the continuation of $\Gamma$, which we denote identically).

\begin{remark}
	If $E$ is complete and metrizable, then $\Gamma$ can be continued to the whole $E\overline{\otimes} E'$ (this may be proved using the same topological arguments as in \cite[Ch. II, \S 1.1]{GroPTT}). In particular, if $E$ is a Banach space, $E\overline{\otimes} E'$ contains only Fredholm kernels, and in this case the class of Fredholm operators and the class of nuclear operators coincide.
\end{remark}

Following the classical approach from \cite{GroPTT}, given a Fredholm kernel $u\in E\overbracket{\otimes} E'$, one may try defining the (quaternionic) trace of the Fredholm operator $\Gamma(u)$ by the identity
\[
T_{\hh,1}(\Gamma(u)):=T_{\hh,1}(u),
\]
with $T_{\hh,1}(u)$ as defined in the previous subsection. However, the mapping $\Gamma$ may not be bijective, in which case the quantity $T_{\hh,1}(\Gamma(u))$ is not well defined.
On the other hand, it is clear from the definitions that if
\begin{equation}
\label{EQuniquenessproperty}
\Gamma(u)=0 \qquad \text{implies}\qquad T_{\hh,1}(u)=0,
\end{equation}
on a linear subspace $Y\subset \Gamma\big( E\overbracket{\otimes} E'\big)$, then, given a Fredholm operator $T\in Y$, the quaternionic trace $T_{\hh,1}(T)=T_{\hh,1}(\Gamma^{-1}(T))$ is well defined. This is related to the so-called \textit{uniqueness problem}, which consists in showing that the trace of an arbitrary Fredholm operator in a (locally convex) space $E$ is well defined (in which case we say that the uniqueness problem has a positive solution in $E$). Obviously, \eqref{EQuniquenessproperty} gives a sufficient condition for the uniqueness problem to have a positive solution in a subspace  $Y \subset E$ (which may be $E$ itself). In the classical noncommutative case, Grothendieck showed in \cite{GroPTT} that the uniqueness problem has a positive solution in a Banach space $E$ if and only if it possesses the approximation property, whilst that a positive solution in a locally convex space $E$ is equivalent to certain topological approximation-type conditions (see also \cite{litvinov,litvinov0}; the assertion is not true in locally convex spaces with the usual approximation property \cite[Theorem 2]{litvinov}).

Grothendieck also found \cite[Ch. II, \S 1]{GroPTT} a linear subspace of
$E\overbracket{\otimes} E'$ in which the mapping $\Gamma$ is bijective. We introduce its quaternionic counterpart.
\begin{definition}\label{DEFpfredholmoperator}
	Let $0<p\leq 1$ and let $E$ be a quaternionic locally convex space. We say that $u\in E\overbracket{\otimes} E'$ is a \textit{$p$-summable Fredholm kernel} on $E$ if $u$ is a Fredholm kernel on $E$, with $\{\mu_k\}\in \ell^p$. The images $\Gamma(u)$ of such kernels  are called \textit{$p$-summable Fredholm operators}. If the sequence $\{x_k'\}$ defining $u$ is equicontinuous, $\Gamma(u)$ is called a \textit{$p$-nuclear operator}.
\end{definition}

It is shown in \cite{GroPTT} that for $p$-summable Fredholm kernels $u$ with $p\leq\frac{2}{3}$, the uniqueness problem has a positive solution. What is more, in this case the  Grothendieck-Lidskii formula holds for the operator $\Gamma(u)$, and the mapping $\Gamma$ is bijective. In what follows, we aim to obtain corresponding statements for the quaternionic case.

\subsection{The quaternionic Grothendieck-Lidskii formula}

\begin{theorem}\label{THMlcs1}
	Let $E$ be a quaternionic locally convex space, and let $u$ be a $p$-summable Fredholm kernel on $E$, where $p\leq 1$. Then, the quaternionic Fredholm determinant of $u$ is an entire function of order $(1/p-1/2)^{-1}$.
\end{theorem}
\begin{proof}
Let us first find expressions for $T_{\hh,m}(u_n)$ for finite-rank  tensors $u_n$ in terms of $\langle x_k',x_\ell\rangle$ rather than directly from the definition \eqref{Higher_order_trace_LCS} (i.e., in terms of $T_{\hh,1}(u_n^d)$, with $d\leq m$). We first find such expressions in the finite-rank case and then pass to the general (infinite-rank) case by limiting arguments. For the truncated kernels
\[
u_n=\sum_{k=1}^n\mu_k x_k\otimes x_k',
\]
it is clear by definition that  the matrix representation of $\Gamma(u_n)$ is
\begin{equation}
\label{EQmatrixlcs}
\begin{pmatrix}
\mu_1 \langle x_1',x_1 \rangle & \mu_1\langle x_1',x_2 \rangle &\cdots &\mu_1\langle x_1',x_n \rangle\\
\mu_2 \langle x_2',x_1 \rangle & \mu_2\langle x_2',x_2 \rangle &\cdots &\mu_2\langle x_2',x_n \rangle\\
\vdots &\vdots & \ddots &\vdots \\
\mu_n \langle x_n',x_1 \rangle & \mu_n\langle x_n',x_2 \rangle &\cdots &\mu_n\langle x_n',x_n \rangle
\end{pmatrix}.
\end{equation}
Writing, for each $k$,
\[
x_k=u_k^1+u_k^2j,\qquad x_k'=v_k^1+v_k^2j,
\]
we have
\[
\langle x_k',x_\ell \rangle = \langle v_k^1,u_\ell^1\rangle  + \langle v_k^2,u_\ell^2\rangle  + \big( \langle v_k^1,u_\ell^2\rangle - \langle v_k^2,u_\ell^1\rangle \big)j=:\alpha_{k,\ell}+ \beta_{k,\ell} j,\qquad \alpha_{k,\ell},\beta_{k,\ell} \in \cc.
\]
We can assume without loss of generality that $|\langle x_k',x_\ell\rangle|\leq 1$ for every $k,\ell$. In this case, it is clear that $|\alpha_{k,\ell}|,|\beta_{k,\ell}|\leq 2$.

We can now associate to $u_n$ the corresponding companion matrix of \eqref{EQmatrixlcs}, i.e., the matrix
\[
\chi_{\Gamma(u_n)}= \begin{pmatrix}
\alpha^{(n)}& \beta^{(n)}\\
-\overline{\beta^{(n)}} & \overline{\alpha^{(n)}}
\end{pmatrix},
\]
with $\alpha^{(n)}=\{\mu_k\alpha_{k,\ell}\}_{k,\ell=1}^n,\beta^{(n)}=\{\mu_k\beta_{k,\ell}\}_{k,\ell=1}^n\in M_n(\cc)$.
For convenience, let us write
\[
\chi_{\Gamma(u_n)}=\{\mu_{p_n(k)} \gamma_{k,\ell}\}_{k,\ell=1}^{2n},
\]
where
\[
p_n(k)=\begin{cases}
k,&\text{if }1\leq k\leq n,\\
k-n, &\text{if }n+1\leq k\leq 2n.
\end{cases}
\]
The quaternionic Fredholm determinant (of the tensor $u_n$) then takes the form
\[
\det{}_{\hh}(I+zu_n) =\sum_{k=0}^{2n}  T_{\hh,k}(u_n)z^k  =\sum_{k=0}^{2n} T_{\hh,k}(\Gamma(u_n))z^k   =\sum_{k=0}^{2n}  \tr\Big( \bigw^{k} \chi_{\Gamma(u_n)}\Big)z^k ,
\]
where, by \eqref{EQformulaextprod},
\[
T_{\hh,k}(u_n)=\tr\Big( \bigw^{k} \chi_{\Gamma(u_n)}\Big) = \sum_{1\leq i_1 <\cdots <i_k\leq 2n} \mu_{p(i_1)}\cdots \mu_{p(i_k)}\det(\gamma_{i_\alpha,i_\beta})_{1\leq \alpha,\beta\leq k}.
\]
We now use the fact that the determinant of a $k\times k$ complex matrix with entries having modulus less than or equal to 1, is bounded from above by $k^\frac{k}{2}$ \cite{Had}. Since $|\gamma_{k,\ell}|\leq 2$ for every $k,\ell$, this implies
\[
\det(\gamma_{i_\alpha,i_\beta})_{1\leq \alpha,\beta\leq k}\leq 2^k k^\frac{k}{2},
\]
for any set of indices $\{i_1,\ldots ,i_k\}$. Thus,
\[
| T_{\hh,k}(u_n) |\leq   2^k k^\frac{k}{2}\sum_{1\leq i_1 <\cdots <i_k\leq 2n} \mu_{p(i_1)}\cdots \mu_{p(i_k)}\leq  2^k k^\frac{k}{2} \sum_{\substack{i_1 \leq \cdots \leq i_k\\\text{multiplicity of }i_k\text{ is }\leq 2}} \mu_{i_1}\cdots \mu_{i_k},
\]
where the last estimate is independent of $n$. Letting $n\to \infty$, we get
\[
|T_{\hh,k}(u)|\leq 2^k k^\frac{k}{2} \sum_{\substack{i_1 \leq \cdots \leq i_k\\\text{multiplicity of }i_k\text{ is }\leq 2}} \mu_{i_1}\cdots \mu_{i_k}.
\]
We now estimate the last sum. First, we note that such a sum corresponds to the (unsigned) coefficient of $z^k$ of the infinite product given by
\[
g(z) = \prod_{k=1}^\infty (1-\mu_kz)^2 = \prod_{k=1}^\infty \big( 1-2\mu_kz+\mu_k^2z^2\big).
\]
Since $\{\mu_k\}\in\ell^p$, we derive from Theorem~\ref{THMinfprod} that $g$ is an entire function of order $p$. Further, by Corollary~\ref{REMcorasymptotic}, we get, for every $q>p$ and $k\in \nn$,
\[
\sum_{\substack{i_1 \leq \cdots \leq i_k\\\text{multiplicity of }i_k\text{ is }\leq 2}} \mu_{i_1}\cdots \mu_{i_k} \leq Ck^{-\frac{k}{q}}.
\]
Hence,
\[
|T_{\hh,k}(u)|\leq C2^k k^{k\big(\frac{1}{2}-\frac{1}{q} \big)}=C2^k k^{-\frac{k}{r}},
\]
where $\frac{1}{r}=\frac{1}{q}-\frac{1}{2}$. Applying  Theorem~\ref{THMorder}, we find that $\det{}_\hh(I+zu)$ is an entire function of order $r$, and so it is of order $s$, where $\frac{1}{s}=\frac{1}{p}-\frac{1}{2}$, since $q>p$ is arbitrary.
\end{proof}

In order to obtain the quaternionic Grothendieck-Lidskii formula, we first need the following result concerning the genus of the quaternionic Fredholm determinant of a $p$-summable Fredholm kernel.
\begin{lemma}\label{LEMgenus}
	Let $E$ be a quaternionic locally convex space, and let $u$ be a $p$-summable Fredholm kernel on $E$, where $0<p\leq 2/3$. Then, the quaternionic Fredholm determinant of $u$ has genus 0.
\end{lemma}
\begin{proof}
	Writing $u$ as
	\[
	u=\sum_{k=1}^\infty (\sqrt{\mu_k}x_k)\otimes ( \sqrt{\mu_{k}}x_k'),
	\]
	with $|\langle x_k',x_\ell\rangle|\leq 1$ for every $k,\ell$, it is clear that the quaternionic Fredholm determinant of $u$ is the same as that of the infinite matrix $A=(a_{k\ell})_{k,\ell=1}^\infty$, with $a_{k\ell}=\sqrt{\mu_k\mu_\ell}\langle x_k',x_\ell\rangle$. Since $\big\{\sqrt{\mu_k}\big\}\in \ell^{4/3}$ and $|\langle x_k',x_\ell\rangle|\leq 1$, the matrix $A$ defines a Hilbert-Schmidt operator in $\ell^2$, by Lemma~\ref{LEMcomposition}. What is more, such an operator is of trace class, since
	\[
	\frac{3}{4}+\frac{3}{4}-\frac{1}{2}=1,
	\]
	and therefore the quaternionic Fredholm determinant of $A$ (and hence, that of $T$), has genus zero, by Corollary~\ref{CORgenuszero}.
\end{proof}
\begin{theorem}\label{THMlidskii}
	Let $E$ be a quaternionic locally convex space, and let $u$ be a $p$-summable Fredholm kernel on $E$, where $0<p\leq 2/3$. Then, the Grothendieck-Lidskii formula holds. More precisely, if $\{\lambda_k(T)\}$ is the sequence of standard eigenvalues of $T:=\Gamma(u)$, one has
	\begin{align*}
	T_{\hh,1}(u)&=2\Re\bigg(\sum_{k=1}^\infty \lambda_k(T)\bigg),\\
	T_{\hh,2}(u)&=\sum_{k=1}^\infty |\lambda_k(T)|^2 +4\sum_{k=1}^\infty\sum_{\ell=k+1}^\infty\Re(\lambda_k(T))\Re(\lambda_\ell(T)).
	\end{align*}
	and
	\[
	\det{}_\hh(I+zu)=\prod_{k=1}^\infty (1+2\Re(\lambda_k(T))z+|\lambda_k(T)|^2z^2).
	\]
\end{theorem}
\begin{proof}
	The equality concerning the Fredholm determinant follows readily from the fact that $\det{}_\hh(I+zu)$ is an entire function of order 1 and genus 0 (by Theorem~\ref{THMlcs1} and Lemma~\ref{LEMgenus}), and its zeros are $\big\{-\lambda_k(T)^{-1}\big\}\cup\big\{ -\overline{\lambda_k(T)}^{-1}\big\}$, together with Hadamard's representation. In order to prove the equalities concerning the first and second-order traces, we further use Hadamard's representation. We may write
	\begin{equation}
	\label{EQrepdeterminant}
	h(z):=\det{}_\hh(I+zu)= e^{a+bz}\prod_{k=1}^\infty(1+z\lambda_k(T))(1+z\overline{\lambda_k(T)}) e^{-2\Re (\lambda_k(T))z},
	\end{equation}
	where, since  $\sum |\lambda_k(T)|<\infty$ (by Theorem~\ref{THMlcs1}) and $h$ is of genus zero, we necessarily have
	\[
	b=2\Re\bigg(\sum_{k=1}^\infty \lambda_k(T)\bigg).
	\]
	From the definition $\det{}_\hh(I+zu)=\sum_{k=0}^\infty z^k T_{\hh,k}(u)$, it follows that $h(0)=1$ and $h'(0)=T_{\hh,1}(u)$. Thus, $a=0$ in \eqref{EQrepdeterminant}. On the other hand, it is easy to see, by applying the product differentation rule to entire functions represented by infinite products to \eqref{EQrepdeterminant}, that $b=T_{\hh,1}(u)$. This establishes the quaternionic Grothendieck-Lidskii trace identity. The corresponding identity for $T_{\hh,2}(u)$ follows simply by inspecting the quadratic term of $\det{}_\hh(I+zu)$.
\end{proof}

\begin{corollary}\label{CORuniqueness}
	Under the same assumptions as in Theorem~\ref{THMlidskii}, the quaternionic trace of $T$ (and hence also the $k$-th order traces as well as the Fredholm determinant of $T$) are well defined, i.e., the quantities
	\[
	\det{}_\hh(I+zT)=\det{}_\hh(I+z\Gamma^{-1}(T)), \qquad T_{\hh,k}(T)=T_{\hh,k}(\Gamma^{-1}(T)), \, k\geq 1,
	\]
	do not depend on the choice of $\{x_k\}$ and $\{x_k'\}$ in \eqref{EQnuclearop}.
\end{corollary}
\begin{proof}
	By Theorem~\ref{THMlidskii}, \eqref{EQuniquenessproperty} holds for the subspace of $\frac{2}{3}$-summable Fredholm operators. Thus, the uniqueness problem has a positive solution in this subspace.
\end{proof}

\begin{remark}
	Corollary~\ref{CORuniqueness} obviously holds for the class of $\frac{2}{3}$-nuclear operators, since they are a subclass of the class of $\frac{2}{3}$-summable Fredholm operators.
\end{remark}

Finally, we derive summability properties of the standard eigenvalues of $p$-summable Fredholm operators. To this end, we first need an auxiliary result for compositions of $p$-summable Fredholm kernels.
\begin{theorem}\label{THMcomposition}
	Let $E$ be a quaternionic locally convex space, and let $u$ and $v$ be $p$ and $q$-summable Fredholm kernels, respectively, where $0<p,q\leq 1$. Then, the composition $v\circ u \in E\overline{\otimes}E'$ is a Fredholm kernel whose Fredholm determinant has order $r\leq 1$ and genus zero, where $\frac{1}{r}=\frac{1}{p}+\frac{1}{q}-1$. Furthermore, the sequence of standard eigenvalues of $\Gamma(v\circ u)$ satisfies $\{\lambda_k(\Gamma(v\circ u))\}\in \ell^r$.
\end{theorem}
\begin{proof}
	It is clear that $v\circ u$ is a Fredholm kernel, since $u$ and $v$ are Fredholm kernels. Now, if $u$ and $v$ are of the form \eqref{EQuv}, writing, for $k\geq 1$,
	\[
	Y_k=\sum_{\ell=1}^\infty \nu_\ell y_\ell \langle y_\ell',x_k \rangle,
	\]
	we may rewrite \eqref{EQuvcomp} as
	\[
	v\circ u= \sum_{k=1}^\infty \mu_k Y_k\otimes x_k' =  \sum_{k=1}^\infty  (\sqrt{\mu_k} Y_k)\otimes (\sqrt{\mu_k}  x_k').
	\]
	Thus, the Fredholm determinant of $v\circ u$ is the Fredholm determinant of the matrix $A=(a_{mn})_{m,n=1}^\infty$, where
	\begin{align*}
		a_{mn}&=\langle \sqrt{\mu_m}x_m' ,\sqrt{\mu_n} Y_n\rangle = \sum_{\ell=1}^\infty \sqrt{\mu_m \mu_n}  \nu_\ell \langle x_m' , y_\ell\rangle \langle y_\ell',x_n\rangle  \\
		&= \sum_{\ell=1}^\infty (\sqrt{\mu_m \nu_\ell}\langle x_m' , y_\ell\rangle ) \cdot (\sqrt{\mu_n \nu_\ell}  \langle y_\ell',x_n\rangle).
	\end{align*}
	Defining the infinite matrices $B=(b_{mn})_{m,n=1}^\infty$ and $C=(c_{mn})_{m,n=1}^\infty$ by
	\[
	b_{mn}= \sqrt{\mu_n \nu_m}\langle x_m',y_n\rangle, \qquad c_{mn}= \sqrt{\mu_m\nu_n }\langle y_m',x_n\rangle,
	\]
	we have $A=CB$. Without loss of generality we may assume that $|\langle x_m',y_n\rangle|,\, |\langle y_m',x_n\rangle|\leq M<\infty$. Then, $B$ and $C$ are matrices representing Hilbert-Schmidt operators in $\ell^2$. Furthermore, by Lemma~\ref{LEMcomposition} (with $2p$ and $2q$ in place of $p$ and $q$, respectively), the matrices $B$ and $C$ actually define an operator in the class $S_t$ acting on the Hilbert space $\ell^2$, with
	\[
	\frac{1}{t}=\frac{1}{2p}+\frac{1}{2q}-\frac{1}{2}.
	\]
	Thus, by H\"older's inequality for operators, $A$ defines an operator in the class $S_r$, where
	\[
	\frac{1}{r}=\frac{2}{t}= \frac{1}{p}+\frac{1}{q}-1.
	\]
	Since $p,q\leq 1$, we have $r\leq 1$. Thus, the conclusion follows from Corollary~\ref{CORgenuszero}.
\end{proof}

\begin{corollary}
	Under the same assumptions as in Theorem~\ref{THMlidskii}, we have $\{\lambda_k(T)\}\in\ell^r$, where
	\[
	\frac{1}{r}= \frac{1}{p}- \frac{1}{2}.
	\]
	In particular, if $u$ is a $\frac{2}{3}$-summable Fredholm kernel, then $\{\lambda_k(T)\}\in\ell^1$.
\end{corollary}
\begin{proof}
	By Theorem~\ref{THMcomposition}, the composition $u^2$ is a Fredholm kernel with Fredholm determinant of order $s$ and genus $0$, where
	\[
	\frac{1}{s}=\frac{2}{p}-1,
	\]
	and moreover $\{\lambda_k(T^2)\}\in \ell^s$. Putting $r=\frac{s}{2}$, since $|\lambda_k(T^2)|=|\lambda_k(T)|^2$, we have that  $\{\lambda_k(T)\}\in \ell^{r}$.
\end{proof}

\section{Appendix: well-definiteness of the quaternionic Fredholm determinant in Hilbert spaces}

The results and discussion in this section are well known in the classical case. However, it is not clear (a priori) whether they should extend to the noncommutative case. For the sake of completeness, we include these results with details, which highlight the differences with the commutative case.

We aim to show that the quaternionic Fredholm determinant in infinite-dimensional Hilbert spaces is well defined for trace-class operators $T$, and that it can be defined as the limit of the Fredholm determinants of a sequence of operators (increasing in rank) converging to $T$ in the trace-class norm.

We start with the Banach algebra inequality
\begin{equation}
\label{EQalgebraprop}
\|AB\|_1\leq \|A\|_1\|B\|_1,
\end{equation}
for quaternionic trace-class operators  $A,B$ on a Hilbert space $H$. To this end, we need the min-max theorem for self-adjoint positive operators (recall that an operator $A\in \mathcal{L}(H)$ is said to be positive if $\langle Ax,x\rangle\geq 0$ for every $x\in H$; a compact self-adjoint operator is positive if and only if all of its eigenvalues are positive, cf. \cite[Ch. III, \S 9]{GG}).
\begin{theorem}
	Let $H$ be a quaternionic Hilbert space and $A\in \mathcal{L}(H)$ be compact and positive. Let $\{\lambda_k(A)\}$ be the sequence of eigenvalues of $A$, in decreasing order. Then, for every $n\in \nn$,
	\[
	\lambda_n(A) = \min_{\dim M=n-1} \max_{\substack{x\perp M\\\|x\|=1}}\langle Ax,x\rangle.
	\]
\end{theorem}
The proof follows exactly the same lines as in the classical case (see \cite[Ch. III, Theorem~9.1]{GG}).
\begin{theorem}
	For any trace-class operators $A,B$ on a quaternionic Hilbert space $H$, the inequality \eqref{EQalgebraprop} holds.
\end{theorem}
\begin{proof}
By definition, we have
\[
\mu_n(AB)^2=\lambda_n(|AB|^2).
\]
We now use that $|AB|^2=(AB)^*AB$ and $|B|^2=B^*B$ are compact positive operators to apply the min-max theorem and obtain
\begin{align*}
\lambda_n(|AB|^2)&= \min_{\dim M=n-1} \max_{\substack{x\perp M\\\|x\|=1}}\langle (AB)^*ABx,x\rangle =\min_{\dim M=n-1} \max_{\substack{x\perp M\\\|x\|=1}}\| ABx\|^2 \\
&\leq \| A\|^2 \min_{\dim M=n-1} \max_{\substack{x\perp M\\\|x\|=1}}\| Bx\|^2= \| A\|^2\min_{\dim M=n-1} \max_{\substack{x\perp M\\\|x\|=1}}\langle B^*Bx,x\rangle\\
&=\| A\|^2  \lambda_n(|B|^2)=\| A\|^2 \mu_n(B)^2.
\end{align*}
Since $\|A\|\leq \|A\|_1$, we get
\[
\mu_n(AB)\leq \|A\|_1 \mu_n(B),
\]
and we obtain the desired inequality by summing up on $n$.
\end{proof}

\begin{lemma}\label{LEMexpsums}
	Let $F$ be a finite-rank operator on a quaternionic Hilbert space. For $z\in \cc$ such that $|z|$ is small enough, we have
	\[
	\det{}_\hh(I-zF)= \exp\bigg(\sum_{m=1}^\infty \frac{(-1)^{m+1}}{m}T_{\hh,1}(F^m)z^m\bigg).
	\]
\end{lemma}
\begin{proof}
	Let $\{\lambda_k\}_{k=1}^n$ be the standard eigenvalues of $F$. By Corollary~\ref{CORfiniterank}, we have
	\begin{align*}
	\det{}_\hh(I-zF)&=\prod_{k=1}^n (1-z\lambda_k)(1-z\overline{\lambda_k}) = \exp\bigg( \sum_{k=1}^n \log (1-z\lambda_k)+ \log(1-z\overline{\lambda_k})\bigg)\\
	&=\exp\bigg( \sum_{k=1}^n \sum_{m=1}^\infty \frac{(-1)^{m+1}}{m} \lambda_k^m z^m  +\frac{(-1)^{m+1}}{m} \overline{\lambda_k^m} z^m\bigg)\\
	&= \exp\bigg( \sum_{k=1}^n \sum_{m=1}^\infty \frac{(-1)^{m+1}}{m}2\Re(\lambda_k^m) z^m \bigg)\\
	&=\exp\bigg(\sum_{m=1}^\infty \frac{(-1)^{m+1}}{m}  2\Re\bigg(\sum_{k=1}^n\lambda_k^m\bigg)z^m\bigg)=\exp\bigg(\sum_{m=1}^\infty \frac{(-1)^{m+1}}{m} T_{\hh,1}(F^m) z^m\bigg),
	\end{align*}
as desired.
\end{proof}
\begin{lemma}\label{LEMdetdiffbound}
	Let $F$ and $G$ be a finite-rank operators on a quaternionic Hilbert space $H$, and assume that $\|F\|_1,\|G\|_1\leq r<1$. Then, there exists a constant $M>0$ such that
	\[
	|\det{}_\hh(I+F)-\det{}_\hh(I+G)|\leq M\|F-G\|_1.
	\]
\end{lemma}
\begin{proof}
By Lemma~\ref{LEMexpsums}, we have, for $r$ small enough,
\begin{align*}
&\phantom{=}|\det{}_\hh(I+F)-\det{}_\hh(I+G)|\\
& = \bigg|\exp\bigg(-\sum_{m=1}^\infty \frac{1}{m}T_{\hh,1}(F^m)\bigg)-\exp\bigg(-\sum_{m=1}^\infty \frac{1}{m}T_{\hh,1}(G^m)\bigg)\bigg|.
\end{align*}
Using Theorem~\ref{THMlpineqhilbert} and the inequality \eqref{EQalgebraprop}, we get
\[
T_{\hh,1}(F^m) \leq 2 \sum_{k=1}^n |\lambda_k(F^m)|\leq 2\sum_{k=1}^n \mu_k(F^m) = 2\|F^m\|_1\leq 2\|F\|_1^m\leq 2r^m,
\]
and similarly for $T_{\hh,1}(G^m)$. By the mean value theorem, we have
\begin{align*}
|\det{}_\hh(I+F)-\det{}_\hh(I+G)| &\leq C_r \bigg| \sum_{m=1}^\infty \frac{1}{m}\big(T_{\hh,1}(F^m)-T_{\hh,1}(G^m)\big)\bigg|\\
&\leq C_r\sum_{m=1}^\infty \frac{1}{m}\big|T_{\hh,1}(F^m)-T_{\hh,1}(G^m)\big|\\
&=C_r\sum_{m=1}^\infty \frac{1}{m}\big|T_{\hh,1}(F^m-G^m)\big|\leq 2C_r\sum_{m=1}^\infty \frac{1}{m}\big\|F^m-G^m\big\|_1,
\end{align*}
where $C_r\leq C<\infty$ uniformly in $r<1$. Let now $q=\max\big\{\|F\|_1,\|G\|_1\big\}\leq r$. It follows easily by induction, \eqref{EQalgebraprop}, and the inequality
\[
\|F^m-G^m\big\|_1 \leq \| (F^{m-1}-G^{m-1})F\|_1+\| G^{m-1}(F-G)\|_1,
\]
 that
\[
\|F^m-G^m\big\|_1 \leq mq^{m-1} \|F-G\|_1,
\]
Thus, we conclude
\[
|\det{}_\hh(I+F)-\det{}_\hh(I+G)| \leq \frac{2C}{1-q}\|F-G\|_1.\qedhere
\]
\end{proof}
\begin{theorem}\label{THMdetindependentchoice}
	Let $T$ be a trace-class operator on a quaternionic Hilbert space $H$, and assume $\{F_n\}$ is a sequence of finite-rank operators on $H$ converging to $T$ in the trace-class norm. Then,
	\[
	\det{}_\hh(I+T)=\lim_{n\to \infty}\det(I+F_n).
	\]
	In particular, the definition of $\det{}_\hh(I+T)$ does not depend on the choice of the sequence $\{F_n\}$.
\end{theorem}
\begin{proof}
	Since $T$ is compact, we may write $T=A+F$, where $F$ is a finite-rank operator and $\|A\|\leq \|A\|_1\leq r<1$. We define another finite-rank operator $G_n=F_n-F$, so that
	\[
	\| A-G_n\|_1=\|T-F-G_n\|_1=\|T-F_n\|_1\to 0, \qquad \text{as }n\to \infty.
	\]
	By Lemma~\ref{LEMdetdiffbound}, it follows that $\{\det{}_\hh(I+G_n)\}$ is a Cauchy sequence, and thus it converges. Indeed, for arbitrarily small $\varepsilon>0$ and $n$ large enough, we have $\| G_n\|_1\leq r+\varepsilon< 1$, and hence
	\[
	|\det{}_\hh(I+G_n)-\det{}_\hh(I+G_m)|\leq M( \|G_n-A\|_1+\|G_m-A\|_1)\to 0,\qquad \text{as }m,n\to \infty.
	\]
	Now, the fact that $\|A\|\leq r$ and $\|A-G_n\|\leq \|A-G_n\|_1\to 0$ implies that the operators $I+A$ and $I+G_n$ are invertible for $n$ large enough. For such $n$, we have
	\begin{equation}
		\label{EQdetlimit}
	\det{}_\hh(I+F_n)=\det{}_\hh(I+G_n+F) =\det{}_\hh(I+G_n)\det{}_\hh(I+(I+G_n)^{-1}F),
	\end{equation}
	where we have applied the well-known identity $\det(A+BC)=\det(A)\det(I+CA^{-1}B)$, taking into account that the involved determinants correspond to complex matrices. In order to show that the sequence $\{\det{}_\hh(I+(I+G_n)^{-1}F)\}$ converges, we write
	\[
	F=\sum_{k=1}^N x_k\otimes x_k'.
	\]
	Since $(I+G_n)^{-1}\to (I+A)^{-1}$ in the operator norm, we get
	\[
	(I+G_n)^{-1}F = \sum_{k=1}^N \big( (I+G_n)^{-1}x_k\big)\otimes x_k'\to  \sum_{k=1}^N \big( (I+A)^{-1}x_k\big)\otimes x_k' , \qquad \text{as }n\to \infty,
	\]
	and therefore $\det{}_\hh(I+(I+G_n)^{-1}F)\to \det{}_\hh(I+(I+A)^{-1}F)$ as $n\to \infty$ (cf. Proposition~\ref{PROPbasisindep} and Remark~\ref{REMdetpolyn}).
	
	Finally, we prove that the definition of $\det{}_\hh(I+T)$ does not depend on the choice of the sequence $\{F_n\}$. Let $\{K_n\}$ be a sequence of finite-rank operators such that $\|T-K_n\|_1\to 0$. A similar argument as above shows that, for $L_n=K_n-F$, we have
	\begin{equation}
		\label{EQlast}
	|\det{}_\hh(I+G_n)-\det{}_\hh(I+L_n)|\leq M( \|G_n-A\|_1+\|L_n-A\|_1)\to 0,\qquad \text{as }n\to \infty,
	\end{equation}
	which, by \eqref{EQdetlimit}, implies that
	\begin{align*}
	&\phantom{=}| \det{}_\hh(I+F_n)-\det{}_\hh(I+K_n)|\\
	&=  |\det{}_\hh(I+G_n)\det{}_\hh(I+(I+G_n)^{-1}F)-\det{}_\hh(I+L_n)\det{}_\hh(I+(I+L_n)^{-1}F)|\to 0,
	\end{align*}
	as $n\to \infty$, since \eqref{EQlast} holds and
	\[
	\det{}_\hh(I+(I+G_n)^{-1}F),\, \det{}_\hh(I+(I+L_n)^{-1}F)\to \det{}_\hh(I+(I+A)^{-1}F),
	\]
	as $n\to \infty$.
\end{proof}

\end{document}